\documentclass{article}

\usepackage{amsfonts, amsthm}
\usepackage{graphicx}
\usepackage{epstopdf}
\usepackage{algorithmic}
\usepackage{enumitem}
\usepackage{amsopn}
\usepackage{bbm}
\usepackage{mathtools}
\usepackage{array,multirow,booktabs}
\usepackage{hyperref}
\hypersetup{colorlinks={true},linkcolor={blue},citecolor=green}
\usepackage{cleveref}
\usepackage[top=1in, bottom=1.25in, left=1.25in, right=1.25in]{geometry}
\numberwithin{equation}{section}
\ifpdf
  \DeclareGraphicsExtensions{.eps,.pdf,.png,.jpg}
\else
  \DeclareGraphicsExtensions{.eps}
\fi

\newcommand{\R}{\mathbbm{R}}
\newcommand{\N}{\mathbbm{N}}

\newcommand{\cG}{\mathcal{G}}
\newcommand{\co}{\text{cor}}
\newcommand{\email}[1]{\protect\href{mailto:#1}{#1}}

\newtheorem{theorem}{Theorem}[section]
\newtheorem{remark}{Remark}[section]
\newtheorem{cor}{Corollay}[section]
\newtheorem{lem}{Lemma}[section]
\newtheorem{prop}{Proposition}[section]

\title{Hyperbolicity-preserving and well-balanced stochastic Galerkin method for shallow water equations\thanks{{\bf Funding:} A.~Narayan was partially supported by NSF DMS-1848508.}}

\author{Dihan Dai\thanks{Department of Mathematics, University of Utah, Salt Lake City, UT 84112 (\email{dai@math.utah.edu}, \email{epshteyn@math.utah.edu}).}
\and Yekaterina Epshteyn \footnotemark[2]
\and Akil Narayan\footnotemark[2] \thanks{Scientific Computing and Imaging (SCI) Institute, University of Utah, Salt Lake City, UT 84112 (\email{akil@sci.utah.edu}).}}

\usepackage{amsopn}
\usepackage{bbm}
\usepackage{mathtools}
\usepackage{array,multirow,booktabs} 

\begin{document}

\maketitle

\begin{abstract}
  A stochastic Galerkin formulation for a stochastic system of balanced or conservation laws may fail to preserve hyperbolicity of the original system. In this work, we develop a hyperbolicity-preserving stochastic Galerkin formulation for the one-dimensional shallow water equations by carefully selecting the polynomial chaos expansion of the nonlinear $q^2/h$ term in terms of the polynomial chaos expansions of the conserved variables. In addition, in an arbitrary finite stochastic dimension, we establish a sufficient condition to guarantee hyperbolicity of the stochastic Galerkin system through a finite number of conditions at stochastic quadrature points.
  Further,  we develop a well-balanced central-upwind scheme for the stochastic shallow water model and derive the associated hyperbolicty-preserving CFL-type condition. The performance of the developed method is illustrated on a number of challenging numerical tests.
  
  \vspace{2em}
  \noindent{\bf Key Words:}  finite volume method, stochastic Galerkin method, shallow water equations, hyperbolic systems of conservation law and balance laws.
  
  \vspace{2em}
  \noindent{\bf AMS subject classifications:} 35L65, 35Q35, 35R60, 65M60, 65M70
\end{abstract}

\section{Introduction}
The classical one-dimensional deterministic Saint-Venant system of
shallow water equations is,
\begin{equation}\label{eq:swed1}
    \begin{aligned}
        &(h)_t+(q)_x = 0,\\
        &(q)_t+\left(\frac{q^2}{h}+\frac{1}{2}gh^2\right)_x = -ghB_x,
    \end{aligned}
\end{equation}
where $h=h(x,t)$ is the water height, $q = q(x,t)$ is the water
discharge, $g$ is the gravitational constant, and $B=B(x)$ is the
time-independent bottom topography. This system was first derived in
\cite{de1871theorie} and since then has been widely used in modeling
the flows whose horizontal scales are significantly larger than 
vertical scales, such as water flows in rivers, lakes and coastal
areas. However, the accuracy and prediction capabilities of shallow
water models depend strongly on the presence of various uncertainties that naturally arise in
measuring or empirically approximating, e.g., the bottom topography data,
or initial and boundary conditions. Hence,  it is important to
consider a stochastic version of the shallow water equations (SWE). In this work we focus on uncertainty that results in \emph{parameterized}
SWE, where parameters are
modeled as random variables. In particular, we study the polynomial chaos
expansion (PCE) strategy, which is very effective when quantities of interest vary
smoothly with respect to the parameters.

There are two widely used classes of methods for addressing uncertainty in (parameterized) partial differential equations using PCE. 
One class, of \emph{non-intrusive} type methods, computes stochastic quantities by
generating an ensemble of solutions of realizations, each of which may be treated as a deterministic problem. Statistical
information is obtained from this ensemble by post-processing the ensemble
solutions. Examples of such methods include Monte-Carlo-type methods that use
randomly selected samples, and the stochastic collocation methods that use
\textit{a priori} pre-selected samples (e.g., \cite{xiu2005high, nobile2008sparse, mishra2012multilevel}).
Since they rely on multiple queries of existing deterministic solvers,
non-intrusive methods are easy to implement and highly parallelizable, but can
result in less accurate approximations than the intrusive type methods.

The other group of methods are \emph{intrusive} methods. Such methods typically
require a substantial rewrite of legacy code and solvers. In the context of PCE
methods, the prototypical intrusive strategy is the stochastic Galerkin (SG)
approach, wherein one replaces an underlying stochastic process with its truncated PCE \cite{wiener1938homogeneous,xiu2002wiener}, and then
forms a system of differential equations via Galerkin projection in stochastic space. As a
consequence,  one derives a new system of partial differential equations whose unknowns are
(time- and space-varying) coefficients of the PCE. Intrusive methods are
projection-based approximations, and thus their accuracy is near-optimal in an $L^2$
sense for static problems.  Discussion on the existing convergence theory for SG
methods can be found, for example  in \cite{babuska2004galerkin,
le2010spectral}. SG methods have been successfully employed for modeling
uncertainty in diffusion models \cite{xiu2009efficient,
eigel2014adaptive}, kinetic equations
\cite{hu2016stochastic,shu2017stochastic}, and conservation and
balanced laws with symmetric Jacobian matrices
\cite{tryoen2010intrusive}.

For hyperbolic systems, such as the SWE, the associated SG system may not be hyperbolic in general \cite{despres2013robust, jin2019study}. Thus, the intrusive SG formulation can result in a system of differential equations of a different class than the original deterministic system. There are currently several efforts to resolve this issue for more general types of equations and to preserve hyperbolicity of the SG system. For quasilinear hyperbolic systems, hyperbolicity can be ensured by multiplying the SG formulation of the system by the left eigenvector matrix of its flux Jacobian matrix \cite{wu2017stochastic}. Unfortunately this transformation results in a non-conservative form and numerical solvers designed
for conservative formulations cannot be applied directly. A recent operator-splitting based approach has been developed for both the Euler equations \cite{chertock2015operator} and the SWE \cite{chertock2015welluq}, where the original systems are split into hyperbolic subsystems whose SG formulations remain hyperbolic. However, this may still lead to complex eigenvalues due to the mismatch in hyperbolicity sets of the subsystems \cite{schlachter2018hyperbolicity}. Another strategy to resolve the hyperbolicity issue of SG formulation is to introduce an appropriate change of variables. For example, the SG system of  balanced/conservation laws in terms of entropic variables can be shown to be hyperbolic \cite{poette2009uncertainty,poette2019contribution}. In addition, an optimization-based method, called the intrusive polynomial moment method (IPMM), was proposed to calculate the PCE of entropic variables given the PCE of the conserved variables \cite{despres2013robust,poette2009uncertainty,poette2019contribution}. However, the optimization problem in IPMM that must be solved for each cell and at each time step can be computationally expensive. There are also strategies that employ Roe variable formulations: In \cite{pettersson2014stochastic,gerster2019hyperbolic,gerster2020entropies}, the flux of the SG system is constructed using Roe variables and the conservative form of the system is preserved. It has been shown that both the SG formulations of the Euler equation \cite{pettersson2014stochastic} and the SWE \cite{gerster2019hyperbolic} in terms of Roe variables are hyperbolic when using a Wiener-Haar expansion. The SG formulation of the isothermal Euler equations in terms of Roe variables is hyperbolic for any basis function under a positive definiteness condition \cite{gerster2019hyperbolic}. However, it can still be expensive to implement the Roe formulation since the PCE of Roe variables need to be calculated by solving both a nonlinear equation and a linear equation.

The SG formulation of the SWE may not be hyperbolic due to the PCE of the nonlinear, non-polynomial term $q^2/h$ \cite{despres2013robust}. This issue can be partially resolved by using the Roe variables and the Wiener-Haar expansion\cite{gerster2019hyperbolic,gerster2020entropies}. In this work, we develop hyperbolicity-preserving SG PCE formulation for the SWE by carefully selecting the PCE of $q^2/h$ term using only the PCE of the conserved variables. Further, we establish a connection between the hyperbolicity of the SG system and the original system. Namely, we show that preserving positivity of the water height 
a finite number of stochastic quadrature points is sufficient to preserving hyperbolicity of the SG formulation of the SWE. In addition, we will present the well-balanced discretization for our SG formulation of SWE, which preserves positivity of the water height at certain quadrature points in the stochastic domain. In this paper, we adopt the filter from \cite{schlachter2018hyperbolicity} to ensure the positivity-preserving property of the algorithm at stochastic quadrature points, which is one ingredient for ensuring hyperbolicity. However, one can go further in filtering. For example, recent work \cite{kusch2020filtered} utilizes a more sophisticated Lasso-regression-based filter to reduce oscillations of the numerical solution at shocks in the spatial domain. \par In this work, we consider central-upwind scheme as an example of the underlying numerical scheme for the stochastic shallow water equations.  However, the main ideas developed in this work are independent of the particular choice of the numerical solver for hyperbolic problems and can be employed with various choices of the numerical schemes for hyperbolic problems. The central Nessyahu-Tadmor schemes, their generalization into higher resolution central schemes and semi-discrete central-upwind schemes are a class of robust Godunov-type Riemann problem-free projection-evolution methods for hyperbolic systems. They were originally developed in \cite{nessyahu1990non,kurganov2000new,kurganov2001semidiscrete}. The family of central-upwind schemes has been successfully applied to problems in science and engineering, and in particular, to deterministic SWE and related models. A second-order central-upwind scheme was first extended to SWE in \cite{kurganov2002central}. However, the scheme did not simultaneously satisfy the positivity-preserving and well-balanced properties. It was improved in \cite{kurganov2007second} where the developed method captures the ``lake-at-rest'' steady state and preserves positivity of the water height. We refer the interested reader to \cite{kurganov2007adaptive,kurganov2007reduction,bryson2011well,chertock2015well,liu2018well,kurganov2018finite} for examples of other closely related works. The numerical scheme developed in this work is mainly based on further extension to stochastic SWE of the framework proposed in \cite{kurganov2001semidiscrete, kurganov2007second}. 

This paper is organized as follows. In \cref{sec:model}, we introduce the stochastic SWE and the SG discretization of the system using a particular choice of the PCE for $q^2/h$. In \cref{sec:hyperbolicity}, we discuss the hyperbolicity of the SG system obtained in \cref{sec:model} and present a sufficient condition to guarantee hyperbolicity of the SG SWE system. In \cref{sec:scheme}, we present a well-balanced central-upwind scheme for the SG SWE model and derive a hyperbolicty-preserving CFL-type condition. In \cref{sec:results}, we illustrate the robustness of the developed numerical scheme with several challenging tests. 

\section{Modeling Stochastic Shallow Water Equations}\label{sec:model}
This section sets up the stochastic SWE problem and introduces notation used in this article. 

\subsection{Stochastic modeling of the SWE}
We consider a complete probability space $(\Omega, \mathcal{F}, P)$, with event space $\Omega$, $\sigma$-algebra $\mathcal{F}$, and probability measure $P$. For $\omega \in \Omega$, a stochastic version of \eqref{eq:swed1} is
\begin{equation}\label{eq:swesg11}
    \begin{aligned}
        &(h(x,t,\omega))_t+(q(x,t,\omega))_x = 0,\\
        &(q(x,t,\omega))_t+\left(\frac{q^2(x,t,\omega)}{h(x,t,\omega)}+\frac{1}{2}gh^2(x,t,\omega)\right)_x = -gh(x,t,\omega)B_x(x,\omega),
    \end{aligned}
\end{equation}
where uncertainty enters the equation through, e.g., a stochastic model of the initial conditions or of the bottom topography $B$. Here, we present a stochastic model of the bottom topography. However, all our results generalize to other models of uncertainty (e.g., in the initial conditions). We model $B$ as a finite-dimensional random field,
\begin{align*}
  B = B(x,\xi) = B_0(x) + \sum_{k=1}^d B_k(x) \xi_k,
\end{align*}
where $\xi = (\xi_1, \ldots, \xi_d)$ is a $d$-dimensional random variable. Such a model can result, for example, from truncation of an infinite-dimensional Karhunen-Lo\'eve decomposition.
Under this model, the stochastic SWE model \eqref{eq:swesg11} can be written as a function of $\xi$,
\begin{equation}\label{eq:swesg1}
    \begin{aligned}
        &(h(x,t,\xi))_t+(q(x,t,\xi))_x = 0,\\
        &(q(x,t,\xi))_t+\left(\frac{q^2(x,t,\xi)}{h(x,t,\xi)}+\frac{1}{2}gh^2(x,t,\xi)\right)_x = -gh(x,t,\xi)B_x(x,\xi),
    \end{aligned}
\end{equation}
which, for the purposes of this paper, forms the continuous model problem for which we seek to compute numerical solutions.


\subsection{Polynomial chaos expansions}
We assume that the random variable $\xi$ has a Lebesgue density $\rho: \R^d \rightarrow \R$. Polynomial chaos expansions (PCE) seek to approximate dependence on $\xi$ by a polynomial function of $\xi$. With $\nu = (\nu_1, \ldots, \nu_d) \in \N_0^d$ a multi-index, then for $\zeta \in \R^d$ we adopt the standard notation,
\begin{align*}
  \zeta^\nu &\coloneqq \prod_{j=1}^d \zeta_j^{\nu_j}, & \zeta^0 = \zeta^{(0,0,\ldots,0)} &= 1.
\end{align*}
We let $\Lambda \subset \N_0^d$ denote any non-empty, size-$K$ finite set of multi-indices. We will assume throughout that $0 = (0,0,\cdots,0) \in \Lambda$. 
Our PCE approximations will take place in a polynomial subspace defined by $\Lambda$:
\begin{align*}
  P_\Lambda &= \mathrm{span} \{ \zeta^\nu\;\; \big|\;\; \nu \in \Lambda\}, & \dim P_\Lambda &= K \coloneqq |\Lambda|.
\end{align*}
We will also need ``powers" of this set, defined by $r$-fold products of $P_\Lambda$ elements:
{\footnotesize
\begin{align}\label{eq:PL}
 P_\Lambda^r &\coloneqq \mathrm{span} \left\{ \prod_{j=1}^r p_j \;\; \big|\;\; p_j \in P_\Lambda, \; j = 1,\ldots, r \right\}, & \dim P_\Lambda^r &\leq \left(\!\!\left(\begin{array}{c} K \\ r \end{array}\right)\!\!\right) = \left(\begin{array}{c} K+r-1 \\ r \end{array}\right),
\end{align}
}%
where the dimension bound results from a combinatoric argument. Note that since $0 \in \Lambda$, then $P^r_\Lambda \subseteq P^s_\Lambda$ for any $r \leq s$.
We will later exercise the notation above for $r = 3$. If $\rho$ has finite polynomial moments of all orders, then 
there is an $L^2_\rho(\R^d)$-orthonormal basis $\{\phi_k\}_{k=1}^\infty$ of $P_\Lambda$, i.e., 
\begin{align}\label{eq:orthocond}
  \langle\phi_k,\phi_\ell\rangle_{\rho} &\coloneqq \int_\R \phi_k(s) \phi_\ell(s) {\rho}(s)d s  = \delta_{k \ell}, & \phi_1(\xi) &\equiv 1,
\end{align}
for all $k, \ell \in \{1, \ldots, K\}$, with the latter identification of $\phi_1$ being an assumption we make without loss since $0 \in \Lambda$.
If $y(x,t,\cdot) \in L^2_\rho(\R)$, then under mild conditions on the probability measure $\rho$ (see \cite{ernst_convergence_2012}) there exists a convergent expansion of $y$ in these basis functions,
\begin{align*}
  y(x,t,\cdot) &\stackrel{L^2_\rho}{=} \sum_{k=1}^\infty \hat{y}_k(x,t) \phi_k(\cdot), 
\end{align*}
where $\hat{y}_{k}(x,t)$ are (stochastic) Fourier coefficients in the basis $\{\phi_k\}_{k \in \N}$, and $\{\phi_\ell\}_{\ell > K}$ are any $L^2_\rho(\R^d)$-orthonormal basis for the orthogonal complement of $P_\Lambda$ in the space of all $d$-variate polynomials. A $K$-term $P_\Lambda$ PCE \emph{approximation} of the stochastic process $y$ is then formed by truncating the summation above to terms in $P_\Lambda$:
\begin{equation}\label{eq:PCEex} 
    y(x,t,\xi) \approx\sum_{k=1}^K \hat{y}_k(x,t)\phi_k(\xi) =:\mathcal{G}_\Lambda[y](x,t,\xi).
\end{equation}
Above, we have defined the linear projection operator $\mathcal{G}_\Lambda : L^2_\rho \rightarrow P_\Lambda$.


\subsection{Operations on Truncated PCE Expansions}
Polynomial statistics of PCE expansions can be computed from a straightforward manipulation of their coefficients. For example, 
\begin{equation}\label{eq:expvar}
  \mathbb{E}[\mathcal{G}_\Lambda[y](x,t,\xi)] = \hat{y}_1(x,t),\quad \text{Var}[\mathcal{G}_\Lambda[y](x,t,\xi)] = \sum_{k=2}^{K}\hat{y}_k^2(x,t),
\end{equation}
where $\mathbb{E}$ is the expectation operator, and $\text{Var}$ is the variance. 
In contrast, computing PCE expansions of nonlinear expressions is more complicated. To calculate the $P_\Lambda$-truncated PCE of the product of two stochastic processes $y(x,t,\xi)$ and $z(x,t,\xi)$, we introduce the notation
\begin{align}\label{eq:mult-assump}
\mathcal{G}_\Lambda[y,z]
  &\coloneqq \mathcal{G}_\Lambda \left[ \mathcal{G}_\Lambda [y]\; \mathcal{G}_\Lambda[z] \right] 
  = \sum_{m=1}^{K}\left(\sum_{k,\ell=1}^K\hat{y}_k\hat{z}_\ell\langle\phi_k\phi_\ell,\phi_m\rangle_{\rho}\right)\phi_m(\xi).
\end{align}
The approximation above defines the \emph{pseudo-spectral product}, which is a widely used strategy for computing PCE expansion products (e.g. \cite{debusschere2004numerical}\cite{gerster2019hyperbolic}). The pseudo-spectral product is an exact projection onto $P_\Lambda$ of the product of two $P_\Lambda$ projections. Such an operation can be cast in linear algebraic terms by considering vectors comprised of the PCE expansion coefficients. Given $y \in P_\Lambda$, we will hereafter let $\hat{y} \in \R^K$ denote its $\phi_k$-expansion coefficients. We now introduce the linear operator $\mathcal{P}: \R^{K} \rightarrow \R^{K\times K}$, 
\begin{align}\label{eq:pmatrix}
  \mathcal{P}(\hat{y}) &\coloneqq \sum_{k=1}^K\hat{y}_k\mathcal{M}_k, &
  \mathcal{M}_k &\in \R^{K \times K}, &
  (\mathcal{M}_k)_{\ell m} &= \langle\phi_k,\phi_\ell\phi_m\rangle_{\rho},
\end{align}
where $\mathcal{M}_k$ is a symmetric matrix for each $k$.
The following properties hold:
\begin{align}\label{eq:pmatrixproperty}
  \mathcal{P}(\hat{y}) &= \begin{pmatrix}\mathcal{M}_1\hat{y}|\mathcal{M}_2\hat{y}|\cdots |\mathcal{M}_K\hat{y} \end{pmatrix}, & 
  \mathcal{P}(\hat{y})\hat{z} &= \mathcal{P}(\hat{z})\hat{y}, & 
  \widehat{\cG_\Lambda[y, z]} = \mathcal{P}(\hat{y})\hat{z}, 
\end{align}
where the last property is due to \eqref{eq:mult-assump}, and allows us to conclude the following.
\begin{lem}\label{lemma:triple-product}
  Let $a(\xi), b(\xi), c(\xi) \in P_{\Lambda}$ have $\phi_j$-expansion coefficients $\hat{a}, \hat{b}, \hat{c} \in \R^K$, respectively. Then  $\left\langle a, b\, c \right\rangle_{\rho} = \hat{a}^T \mathcal{P}(\hat{b}) \hat{c}$.
\end{lem}
\begin{proof}
  Since $a \in P_\Lambda$, then 
  \begin{align*}
    \left\langle a, b\, c \right\rangle_{\rho} =  \left\langle b\, c, a \right\rangle_{\rho} = \left\langle \cG_\Lambda[b,c], a \right\rangle_{\rho} = \hat{a}^T \widehat{\cG_\Lambda[b, c]} \stackrel{\eqref{eq:pmatrixproperty}}{=} \hat{a}^T \mathcal{P}(\hat{b})\hat{c}.
  \end{align*}
\end{proof}

We will also need to compute $P_\Lambda$ truncations of ratios of processes (when for each $(x,t)$ the denominator is a single-signed process with probability 1). We start by noting the following exact representation when $y$ is a single-signed process:
\begin{equation}\label{eq:prod-assump}
    \mathcal{G}_\Lambda\left[y\,\frac{z}{y}\right](x,t,\xi) =\mathcal{G}_\Lambda[z](x,t,\xi).
\end{equation}
We then use this to motivate the assumption,
\begin{align}\label{eq:div-assump}
    \mathcal{G}_\Lambda\left[y,\frac{z}{y}\right] = \mathcal{G}_\Lambda[z]
    \hskip 5pt \stackrel{\eqref{eq:pmatrixproperty}}{\Longleftrightarrow} \hskip 5pt
    \mathcal{P}(\hat{y}) \widehat{\left(\frac{z}{y}\right)} = \hat{z}.
\end{align}
This expression motivates the following definition for a new operator $\mathcal{G}^{\dagger}_\Lambda\left[\frac{z}{y}\right]$:
\begin{equation}
    \mathcal{G}^{\dagger}_\Lambda\left[\frac{z}{y}\right](\xi) \coloneqq \sum_{k=1}^{K}c_k\phi_k(\xi),
\end{equation}
  where $c_i$ is the $i$th element of $\widehat{\left(\frac{z}{y}\right)}$ defined by \eqref{eq:div-assump}, 
  assuming $\mathcal{P}(\hat{y})$ is invertible. 

\subsection{Stochastic Galerkin Formulation for Shallow Water Equations}
We start with \eqref{eq:swesg1} and perform a standard Galerkin procedure in stochastic ($\xi$) space using polynomials from $P_\Lambda$. I.e., the first step is to replace $h$ and $q$ by the ansatz,
\begin{align}\label{eq:sg-ansatz}
  h \simeq h_\Lambda &\coloneqq \sum_{k=1}^K \hat{h}_j(x,t) \phi_j(\xi), &
  q \simeq q_\Lambda &\coloneqq \sum_{k=1}^K \hat{q}_j(x,t) \phi_j(\xi), 
\end{align}
respectively, and $B$ by $\cG_\Lambda[B]$. Following this, we apply the projection operator $\cG_\Lambda$ to both sides of \eqref{eq:swesg1} and insist on equality. However, in addition we make the following crucial assumption about how we approximate the term $q^2/h$,
\begin{align*}
  \frac{q^2}{h} = \frac{q}{h}\; q \hskip 10pt \longrightarrow \hskip 10pt 
  \cG_\Lambda\left[ \frac{q_\Lambda^2}{h_\Lambda}\right] = \cG_\Lambda \left[  q_\Lambda\; \cG^\dagger_\Lambda\left[\frac{q_\Lambda}{h_\Lambda}\right]\right]
\end{align*}
Performing these steps on \eqref{eq:swesg1} results in the system,
\begin{equation}\label{eq:swesg4}
\dfrac{\partial}{\partial t}\begin{pmatrix}\hat{h}\\\hat{q}\end{pmatrix}
+\frac{\partial}{\partial x}\begin{pmatrix}\hat{q}\\\frac{1}{2}g\mathcal{P}(\hat{h})\hat{h}+\mathcal{P}(\hat{q})\mathcal{P}^{-1}(\hat{h})\hat{q}\end{pmatrix}
=\begin{pmatrix}0\\-g\mathcal{P}(\hat{h})\widehat{B}_x\end{pmatrix},
\end{equation}
where $\hat{h}$ and $\hat{q}$ are each length-$K$ vectors whose entries are the coefficients introduced in \eqref{eq:sg-ansatz}.
With $\hat{U} \coloneqq (\hat{h}, \hat{q})^T$, and the flux and source terms
\begin{align}\label{eq:sweflt}
  F(\hat{U}) &= \begin{pmatrix}\hat{q}\\\frac{1}{2}g\mathcal{P}(\hat{h})\hat{h}+\mathcal{P}(\hat{q})\mathcal{P}^{-1}(\hat{h})\hat{q}\end{pmatrix}, &
    S(\hat{U},\hat{B})&=\begin{pmatrix}0\\-g\mathcal{P}(\hat{h})\widehat{B}_x\end{pmatrix},
\end{align}
then the system \eqref{eq:swesg4} can be written in general conservation law form,
\begin{equation}\label{eq:swesg5}
    \hat{U}_t+(F(\hat{U}))_x = S(\hat{U},\hat{B}),
\end{equation}
with flux Jacobian 
\begin{equation}\label{eq:jacobian1}
  J(\hat{U}) \coloneqq \frac{\partial F}{\partial\hat{U}}=\begin{pmatrix}O&I\\g\mathcal{P}(\hat{h})-\mathcal{P}(\hat{q})\mathcal{P}^{-1}(\hat{h})\mathcal{P}(\hat{u})&\mathcal{P}(\hat{u})+\mathcal{P}(\hat{q})\mathcal{P}^{-1}(\hat{h})\end{pmatrix},
\end{equation}
where we have introduced
\begin{equation}\label{eq:uPCE}
    \hat{u} = \mathcal{P}^{-1}(\hat{h})\hat{q},
\end{equation} 
which can be viewed as the PCE coefficient vector of the velocity $u:=\frac{q}{h}$. The computation that gives the expression \eqref{eq:jacobian1} for the Jacobian uses the property \eqref{eq:pmatrixproperty}. For more details, we refer interested readers to section 2.2 of \cite{jin2019study}. 

We emphasize that $(h,q)$ are the $(x,t,\xi)$-dependent solutions to the original stochastic SWE equations \eqref{eq:swesg1}, whereas $(h_\Lambda, q_\Lambda)$ are the $(x,t,\xi)$-dependent solutions to our SGSWE equations \eqref{eq:swesg5}. In general, these two solutions are distinct. We first articulate sufficient conditions under which \eqref{eq:swesg5} is a well-posed hyperbolic system.

\section{Hyperbolicity of The SG System}\label{sec:hyperbolicity}
In this section we show that the system \eqref{eq:swesg5} is hyperbolic under the condition that the matrix $\mathcal{P}(\hat{h})$ is positive definite. When there is no uncertainty, this condition reduces to $h>0$, which ensures  hyperbolicity for the deterministic shallow water equations \eqref{eq:swed1}.
\begin{theorem}\label{thm:hyperbolicity}
If the matrix $\mathcal{P}(\hat{h})$ is strictly positive definite, the SG formulation \eqref{eq:swesg5} is hyperbolic.
\end{theorem}
\begin{proof}
  We will show that the Jacobian $\frac{\partial F}{\partial \hat{U}}$ is diagonalizable with real eigenvalues. Since $\mathcal{P}(\hat{h})$ is positive definite, then define
  \begin{align}\label{eq:def-abg}
    G &\coloneqq \sqrt{g \mathcal{P}(\hat{h})}, & 
    A &\coloneqq g G^{-1} \mathcal{P}(\hat{q}) G^{-1}, & 
    B &\coloneqq \mathcal{P}(\hat{u}),
  \end{align}
  where $\sqrt{M}$ is the (unique) symmetric positive definite square root of a symmetric positive definite matrix $M$. Using these matrices, define
  \begin{align*}
    P_1 &\coloneqq \begin{pmatrix}I&I\\B+G&B-G\end{pmatrix}, & 
    P_1^{-1} = \left(-\dfrac{1}{2}\right)\begin{pmatrix}G^{-1}B-I&-G^{-1}\\-G^{-1}B-I&G^{-1}\end{pmatrix},
  \end{align*}
  where the formula for $P_1^{-1}$ can be verified by direct computation. Then a calculation shows that
  \begin{align}\label{eq:dfdu-similar}
    P_1^{-1}\frac{\partial F}{\partial \hat{U}}P_1 &= -\frac{1}{2} \left( \begin{array}{cc} -2 G - B - A & A - B \\ A - B & 2 G - B - A \end{array}\right),
  \end{align}
  which is symmetric. Thus $\frac{\partial F}{\partial \hat{U}}$ is similar to a diagonalizable matrix with real eigenvalues, and so is itself real diagonalizable.
%
\end{proof}
\begin{remark}
  In the deterministic case, i.e, all the PCE coefficients are zero except possibly the very first coefficient and the matrix in \eqref{eq:dfdu-similar} reduces to the eigenmatrix that symmetrizes the deterministic Jacobian matrix and a diagonal matrix.
\end{remark}
For the deterministic SWE \eqref{eq:swed1}, the velocity $u$ is bounded between the smallest and the largest eigenvalues of the Jacobian of the deterministic SWE. For the SG formulation \eqref{eq:swesg4}, we have an analogous relation. 
\begin{prop}\label{lem:spectralr}
    The eigenvalues of the matrix $\mathcal{P}(\hat{u})$ are bounded between the smallest and the largest eigenvalues of the Jacobian matrix $J(\hat{U})$, i.e.,
    \begin{equation}\label{ju-and-pu}
        \lambda_{\max}(J(\hat{U}))\ge\lambda_{\max}\left(\mathcal{P}(\hat{u})\right)\ge\lambda_{\min}\left(\mathcal{P}(\hat{u})\right)\ge\lambda_{\min}(J(\hat{U})).
    \end{equation}
\end{prop}
\begin{proof}
  By the proof of \cref{thm:hyperbolicity}, the matrix $J(\hat{U})$ is similar to the symmetric matrix $D \coloneqq P_1^{-1}\frac{\partial F}{\partial \hat{U}}P_1$ defined in \eqref{eq:dfdu-similar}.
%
For an arbitrary unit vector $\hat{y} = \left(\hat{y}_1,\hat{y}_2,\cdots,\hat{y}_K\right)^{\mathrm{T}}\in\mathbb{R}^{K}$, then $\hat{z} \coloneqq \frac{1}{\sqrt{2}}[ \hat{y}^T, \hat{y}^T]^T \in \R^{2 K}$ is also a unit vector. Then,
\begin{equation}
\hat{z}^{\mathrm{T}}D\hat{z} = \hat{y}^{\mathrm{T}}\mathcal{P}(\hat{u})\hat{y}.
\end{equation}
  From the above relation, and using properties of the Rayleigh quotient for $\mathcal{P}(\hat{u})$, 
\begin{equation*}
    \lambda_{\max}(\mathcal{P}(\hat{u}))\ge\hat{z}^{\mathrm{T}}D\hat{z}\ge\lambda_{\min}(\mathcal{P}(\hat{u})),
\end{equation*}
where equalities can be achieved by proper selections of $\hat{y}$. Using similar Rayleigh quotient properties for $D$ and noting that $\hat{z}$ ranges over a subset of $\R^{2 K}$, then 
\begin{equation} \lambda_{\max}(D)\ge\lambda_{\max}\left(\mathcal{P}(\hat{u})\right)\ge\lambda_{\min}\left(\mathcal{P}(\hat{u})\right)\ge\lambda_{\min}(D)
\end{equation}
The inequalities \eqref{ju-and-pu} follow since $D$ is similar to $J(\hat{U})$. 
\end{proof}

In the deterministic SWE, positivity of the water height $h$ ensures hyperbolicity of the PDE system. \cref{thm:hyperbolicity} shows that the stochastic variant of the positivity condition is that $\mathcal{P}(\hat{h})$ is positive definite. Much of the rest of this paper is devoted to deriving numerical procedures to guarantee this condition.

\subsection{Positive definiteness of $\mathcal{P}(\hat{h})$}
In this subsection, we present a computationally convenient sufficient condition that guarantees $\mathcal{P}(\hat{h})>0$, and hence guarantees hyperbolicity. 
\begin{theorem}\label{thm:h-positivity}
 Given $\Lambda$, let nodes $\xi_m$ and weights $\tau_m$ satisfying $\{(\xi_m, \tau_m)\}_{m=1}^M \subset \R^d \times (0, \infty)$ represent any $M$-point positive quadrature rule that is exact on $P_{\Lambda}^3$, i.e., 
  \begin{align}\label{eq:P3-exactness}
    \int_{\R^d} p(\xi) \rho(\xi) d \xi &= \sum_{m=1}^M p(\xi_m) \tau_m, & p &\in P_\Lambda^3.
  \end{align}
  If 
  \begin{align}\label{eq:h-positivity}
    h_\Lambda(x,t,\xi_m) > 0 \;\; \forall\; m = 1, \ldots, M,
  \end{align}
  then the SGSWE system \eqref{eq:swesg5} is hyperbolic.
\end{theorem}
\begin{proof}
  We will show that \eqref{eq:h-positivity} implies $\mathcal{P}(\hat{h}) > 0$, which in turn ensures hyperbolicity from \cref{thm:hyperbolicity}. Let $\hat{z} = \left(\hat{z}_k\right)_{k=1}^K$ be any nontrivial vector in $\R^K$, and define its associated $P_\Lambda$ polynomial $z(\xi) \coloneqq \sum_{k=1}^K \hat{z}_j \phi_k(\xi) \neq 0$.
  Then $z(\xi)$ cannot vanish at all quadrature points simultaneously since if it did we obtain the contradiction,
  \begin{equation*}
    0 \neq \| \hat{z}\|^2 = \left\langle z, z \right\rangle_\rho \stackrel{\eqref{eq:P3-exactness}}{=} \sum_{j=1}^M z^2(\xi_j) \tau_j = 0,
  \end{equation*}
  where we have used the fact that $P^2_\Lambda \subseteq P^3_\Lambda$ to utilize \eqref{eq:P3-exactness}.
  Then since the quadrature rule is positive and \eqref{eq:h-positivity} holds, we have
  \begin{align*}
    0 < \sum_{j=1}^M h_\Lambda(x,t,\xi_j) z^2(\xi_j) \tau_j \stackrel{\eqref{eq:P3-exactness}}{=} \left\langle h_\Lambda(x,t,\xi), z^2(\xi) \right\rangle \stackrel{\textrm{Lemma }\eqref{lemma:triple-product}}{=} \hat{z}^T \mathcal{P}(\hat{h}) \hat{z},
  \end{align*}
  establishing that $\mathcal{P}(\hat{h})$ is positive definite.
\end{proof}
Thus, by guaranteeing positivity of $h_\Lambda$ at a finite number of points, we can ensure hyperbolicity of the SGSWE system. For arbitrary stochastic dimension $d$ and polynomial space $P_\Lambda$, there is a worst-case upper bound on the size of this finite set.
\begin{cor}\label{cor:tchakaloff}
  There is some 
  $M \leq \dim P^3_\Lambda \leq \frac{K(K+1)(K+2)}{6}$ such that the discrete pointwise positivity condition \eqref{eq:h-positivity} guarantees hyperbolicity of \eqref{eq:swesg5}.
\end{cor}
We give the proof in \cref{lemma:tchakaloff} in the Appendix.
%
One might consider the somewhat simpler condition of restricting $\hat{h}_1>0$ for hyperbolicity since $\hat{h}_1$ is the expected value of $h_\Lambda$. This condition is actually implied by the condition in \cref{thm:h-positivity}.
\begin{cor}
  If the conditions of \cref{thm:h-positivity} are satisfied, then $\hat{h}_1 > 0$.
\end{cor}
\begin{proof}
  Since $\tau_j > 0$ and $h_{\Lambda} > 0$ at the quadrature points, then
  \begin{align*}
    \hat{h}_1 = \int_{\R^d} h_\Lambda(x,t,\zeta) \rho(\zeta) d\zeta = \sum_{j=1}^M h_\Lambda(x,t,\xi_j) \tau_j > 0,
  \end{align*}  
\end{proof}

A computable condition ensuring hyperbolicity therefore requires a positive quadrature rule that is exact on $P_{\Lambda}^3$. For general densities $\rho$ over $\R^d$, computing such a quadrature rule is a very difficult task. But this is possible in specialized cases.

For example, if $d = 1$ and $\Lambda = \{0, 1, \ldots, K-1\}$, then an optimal choice of positive quadrature is the $\rho$-Gaussian quadrature. Since $P_\Lambda^3 = \mathrm{span}\{1, \zeta, \ldots, \zeta^{3K-3}\}$, then choosing the positive $M$-point Gaussian quadrature,
\begin{align*}
  \{\xi_m\}_{m=1}^M &= \phi_{M+1}^{-1}(0), & \tau_m &= \frac{1}{\sum_{j=1}^M \phi_j^2(\xi_m)},
\end{align*}
with $M \geq \left\lceil \frac{3 K}{2} \right\rceil - 1$ satisfies the conditions of \cref{thm:h-positivity} (and does so with substantially fewer points than the $\sim K^3/6$ worst-case bound from \cref{cor:tchakaloff}). Gaussian quadrature rules have real-valued nodes and positive weights \cite{szego_orthogonal_1975}.

In spaces with $d > 1$, if $\rho$ is tensorial, then tensorizing Gauss quadrature rules achieves similar results. I.e., assume
\begin{align*}
  \rho(\xi) &= \prod_{J=1}^d \rho_J(\xi_J), & \xi &\in \R^d,
\end{align*}
We can always enclose $P_\Lambda$ within a tensor-product polynomial space:
\begin{align*}
  P_{\Lambda}^3 \subseteq P_{3 k,\infty} &\coloneqq \left\{ \lambda \in \N_0^d \;\big|\; \lambda_J \leq 3 \kappa_J \textrm{ for } J = 1, \ldots, d \right\}, & \kappa_J &\coloneqq \max_{\nu \in \Lambda} \nu_J.
\end{align*}
For a fixed $J \in \{1, \ldots, d\}$, let $\{(\xi^{(J)}_{m,M_J}, \tau^{(J)}_{m,M_J})\}_{m=1}^{M_J}$ denote the $M_J \coloneqq (\left\lceil \frac{3 \kappa_J}{2} \right\rceil - 1)$-point $\rho_J$-Gaussian quadrature rule on $\R$. Then the tensorization of these $d$ univariate quadrature rules results in an $M \coloneqq \left(\prod_{J=1}^d M_J\right)$-point positive quadrature rule that is exact on $P_{3 k, \infty}$, hence on $P_\Lambda^3$, and thus satisfies the conditions of \cref{thm:h-positivity}.



\section{Numerical Scheme for Stochastic Shallow Water Equations}\label{sec:scheme}
In this section, we derive a well-balanced central-upwind scheme that preserves the hyperbolicity of the SG formulation \eqref{eq:swesg5} at every time step. 

\subsection{Central-Upwind Scheme for the SG System}
We first introduce the central-upwind scheme for the SG system \eqref{eq:swesg5}. \cref{append:cuscheme} provides a brief summary of the second-order central-upwind schemes for balance laws.  With $\{\mathcal{C}_i\}_{i=1}^N$ a partition of a bounded closed interval, let $x_{i \pm \frac{1}{2}}$ denote the partition boundaries, and define the cell average of the vector $\hat{U}$ over the $i$th cell $\mathcal{C}_i \eqqcolon \left[ x_{i-\frac{1}{2}}, x_{i+\frac{1}{2}}\right]$ as,
$$
\overline{\mathbf{U}}_i(t) \coloneqq \begin{pmatrix}\overline{\mathbf{h}}_i(t)\\\overline{\mathbf{q}}_i(t)\end{pmatrix}\coloneqq\frac{1}{\Delta x}\int_{\mathcal{C}_i}\begin{pmatrix}\hat{h}(x,t)\\\hat{q}(x,t)\end{pmatrix}dx \in \R^{2 K}.
$$ 
  We have introduced notation for common quantities in finite volume-type schemes. While $\hat{U}_k$ is the $k$th component of the vector $\hat{U}$, the bold letter $\mathbf{U}$ with subscripts and superscripts is used here to introduce the cell averages and pointwise reconstructions, respectively, of the vector $\hat{U}(x,t)$. I.e., $\mathbf{U}_{i+\frac{1}{2}}^-$ is the approximated value of $\hat{U}$ at the left-hand side of spatial location $x = x_{i+\frac{1}{2}}$, which is reconstructed from the cell averages $\overline{\mathbf{U}}_i$.  
  A similar reasoning applies to $(\mathbf{h}, \hat{h}, \hat{h}_k)$ and $(\mathbf{q}, \hat{q}, \hat{q}_k)$. 
To minimize clutter, we will notationally suppress $t$ dependence from here onward. The possible discontinuities of the system \eqref{eq:swesg5} at the cell interface $x = x_{i+\frac{1}{2}}$, where $\mathcal{C}_i = \left[ x_{i-\frac{1}{2}}, x_{i+\frac{1}{2}}\right]$, propagates with left- and right-sided local speeds that can be estimated by,
\begin{equation}\label{eq:pspeed}
    \begin{aligned}
        &a^{-}_{i+\frac{1}{2}} = \min\left\{\lambda_1\left(J(\mathbf{U}_{i+\frac{1}{2}}^{-})\right),\lambda_1\left(J(\mathbf{U}_{i+\frac{1}{2}}^{+})\right),0\right\},\\
        &a^{+}_{i+\frac{1}{2}} = \max\left\{\lambda_{2K}\left(J(\mathbf{U}_{i+\frac{1}{2}}^{-})\right),\lambda_{2K}\left(J(\mathbf{U}_{i+\frac{1}{2}}^{+})\right),0\right\},\\
    \end{aligned}
\end{equation}
where $\lambda_1\le\lambda_2\le\cdots\le\lambda_{2K}$ are the eigenvalues of the $J(\cdot)$ in \eqref{eq:jacobian1}, and 
$\mathbf{U}_{i+\frac{1}{2}}^{-}$ and $\mathbf{U}_{i+\frac{1}{2}}^{+}$ are the left- and right-sided pointwise reconstructions in the $i$th cell. The semi-discrete form of the central-upwind scheme for the SG system \eqref{eq:swesg5} reads as,
\begin{align}\label{eq:semidiscretewsg}
  \dfrac{d}{dt}\overline{\mathbf{U}}_i &= -\dfrac{\mathcal{F}_{i+\frac{1}{2}}-\mathcal{F}_{i-\frac{1}{2}}}{\Delta x}+\overline{\mathbf{S}}_i, &
  \overline{\mathbf{S}}_i&\approx\frac{1}{\Delta x}\int_{\mathcal{C}_i}S(\mathbf{U},\mathbf{B})dx
\end{align}
 with $\overline{\mathbf{S}}_i$ 
 a well-balanced discretization of the source term, which we discuss below. With $F$ the flux term in \eqref{eq:sweflt}, the numerical flux $\mathcal{F}$ is given by
\begin{equation}\label{eq:fluxcusg}
    \mathcal{F}_{i+\frac{1}{2}} \coloneqq \dfrac{a^{+}_{i+\frac{1}{2}}F(\mathbf{U}_{i+\frac{1}{2}}^{-})-a^{-}_{i+\frac{1}{2}}F(\mathbf{U}_{i+\frac{1}{2}}^{+})}{a^{+}_{i+\frac{1}{2}}-a^{-}_{i+\frac{1}{2}}}+\dfrac{a^{+}_{i+\frac{1}{2}}a^{-}_{i+\frac{1}{2}}}{a^{+}_{i+\frac{1}{2}}-a^{-}_{i+\frac{1}{2}}}\left[\mathbf{U}_{i+\frac{1}{2}}^{+}-\mathbf{U}_{i+\frac{1}{2}}^{-}\right].
\end{equation}

\subsection{Well-Balanced Property}\label{sect:wb}
In applications of the deterministic SWE, simulations should accurately capture the so-called ``lake-at-rest'' steady state solution, or small perturbations of the lake-at-rest steady state. A \textit{well-balanced} numerical scheme for the SWE captures the lake-at-rest solution exactly at discrete level. An analogous lake-at-rest state for the stochastic shallow water equations \eqref{eq:swesg4} is
\begin{equation}\label{eq:lake-at-rest}
  q_\Lambda(x,t,\xi) \equiv 0,\quad h_\Lambda + \mathcal{G}_\Lambda[B](x,t,\xi) \equiv C(\xi),
\end{equation}
where $C(\xi)$ depends only on $\xi$. This solution corresponds to still water with a flat stochastic water surface. Equation \eqref{eq:lake-at-rest} can be rewritten in the vector form,
\begin{equation}
\hat{q} \equiv \mathbf{0},\quad\hat{h}+\hat{B} \equiv \hat{C}.
\end{equation}
In order to derive a well-balanced central upwind scheme for the SGSWE, we first replace the original bottom function $\hat{B}$ by its continuous linear interpolant. At every time step, we compute the PCE vector for the cell averages of the water surface by $\overline{\mathbf{w}}_i\coloneqq \overline{\mathbf{h}}_i+\overline{\mathbf{B}}_i$ and the pointwise reconstructions of the water surface by $\mathbf{w}^{\pm}_{i+\frac{1}{2}}$ using a generalized minmod limiter (see \cref{append:cuscheme}). The pointwise reconstructions of the water height are then computed by
\begin{align}\label{eq:hreconstruction}
    \mathbf{h}^{\pm}_{i+\frac{1}{2}}\coloneqq \mathbf{w}^{\pm}_{i+\frac{1}{2}}-\mathbf{B}_{i+\frac{1}{2}},
\end{align}
where $\mathbf{B}_{i+\frac{1}{2}}$ is the PCE vector for $\mathcal{G}_\Lambda\left[ B(x_{i+\frac{1}{2}}, t, \xi)\right]$. The numerical fluxes $\{\mathcal{F}_{i+\frac{1}{2}}\}_{i=1}^{N}$ are subsequently computed using the reconstructed PCE of the water height defined in \eqref{eq:hreconstruction}. 
After that, the well-balanced property of the scheme is ensured by a special choice of the source term $\overline{\mathbf{S}}_i$.
\begin{lem}\label{lem:wbproperty}
  With $\mathbf{B}_{i\pm\frac{1}{2}}$ the PCE vectors for $\mathcal{G}_\Lambda\left[ B(x_{i\pm\frac{1}{2}}, t, \xi)\right]$,\\ if we choose
  \begin{equation}\label{eq:balancedterm}
      \overline{\mathbf{S}}_i\coloneqq \begin{pmatrix}\mathbf{0}\\-\frac{1}{\Delta x}g\mathcal{P}(\overline{\mathbf{h}}_i)\left(\mathbf{B}_{i+\frac{1}{2}}-\mathbf{B}_{i-\frac{1}{2}}\right)\end{pmatrix},
  \end{equation}
  then the central-upwind scheme \eqref{eq:semidiscretewsg} satisfies the well-balanced property.
\end{lem}
\begin{proof}
 We have $\overline{\mathbf{B}}_i = (\mathbf{B}_{i+\frac{1}{2}}+\mathbf{B}_{i-\frac{1}{2}})/2$, and the cell average PCE vector of the water surface $\overline{\mathbf{w}}_{i}\coloneqq \overline{\mathbf{h}}_{i}+\overline{\mathbf{B}}_{i}$. Let the pointwise reconstructions for water surface be $\mathbf{w}^{\pm}_{i+\frac{1}{2}}$. Assume that at time $t$, the stochastic water surface is flat and the water is still, i.e.,  $\overline{\mathbf{w}}_{i}\equiv\mathbf{w^*}$ is a constant vector for all $i$, and $\overline{\mathbf{q}}_i\equiv \mathbf{0}$. Then a second-order piecewise linear reconstruction procedure produces $\mathbf{w}^{\pm}_{i+\frac{1}{2}}\equiv \mathbf{w^*}$ and $\mathbf{q}^{\pm}_{i+\frac{1}{2}}\equiv \mathbf{0}$.
 Hence,  the numerical flux defined in \eqref{eq:fluxcusg} becomes,
\begin{equation}\label{eq:sgflux}
\mathcal{F}_{i+\frac{1}{2}} = \begin{pmatrix}\mathbf{0}\\\frac{g}{2}\mathcal{P}(\mathbf{w^*}-\mathbf{B}_{i+\frac{1}{2}})(\mathbf{w^*}-\mathbf{B}_{i+\frac{1}{2}})\end{pmatrix}=:\begin{pmatrix}\mathcal{F}^{\hat{h}}_{i+\frac{1}{2}}\\\mathcal{F}^{\hat{q}}_{i+\frac{1}{2}}\end{pmatrix}.
\end{equation}
  Then with $\overline{\mathbf{S}}_i = \left( \overline{\mathbf{S}}_{i,1}^T, \overline{\mathbf{S}}_{i,2}^T \right)^T$, the corresponding semidiscrete form is 
  {\small 
\begin{equation}\label{eq:wbp1}
    \begin{aligned}
        \dfrac{d }{dt}\overline{\mathbf{h}}_i &= \overline{\mathbf{S}}_{i,1} \\ 
        \dfrac{d }{dt}\overline{\mathbf{q}}_i &= 
        -\dfrac{1}{\Delta x}\frac{g}{2}\left[\mathcal{P}(\mathbf{w^*}-\mathbf{B}_{i+\frac{1}{2}})(\mathbf{w^*}-\mathbf{B}_{i+\frac{1}{2}})-\mathcal{P}(\mathbf{w^*}-\mathbf{B}_{i-\frac{1}{2}})(\mathbf{w^*}-\mathbf{B}_{i-\frac{1}{2}})\right]+\overline{\mathbf{S}}_{i,2}
    \end{aligned}.
\end{equation}
  }
  To balance these equations, we choose $\overline{\mathbf{S}}_{i,1}$ and $\overline{\mathbf{S}}_{i,2}$ so that the right-hand side vanishes. Clearly we need $\overline{\mathbf{S}}_{i,1} \equiv \mathbf{0}$. 
  To simplify the computation for $\overline{\mathbf{S}}_{i,2}$, let $\Delta \mathbf{B}_i = \mathbf{B}_{i+\frac{1}{2}}-\mathbf{B}_{i-\frac{1}{2}}$, then $\overline{\mathbf{B}}_i = \mathbf{B}_{i+\frac{1}{2}} - \frac{1}{2}\Delta \mathbf{B}_i = \mathbf{B}_{i-\frac{1}{2}} + \frac{1}{2}\Delta \mathbf{B}_i$. By linearity of the operator $\mathcal{P}$ and the property \eqref{eq:pmatrixproperty}, 
\begin{equation}\label{eq:wbp3}
    \begin{aligned}
      \overline{\mathbf{S}}_{i,2} &= \dfrac{1}{\Delta x}\frac{g}{2}\left[\mathcal{P}(\mathbf{w^*}-\mathbf{B}_{i+\frac{1}{2}})(\mathbf{w^*}-\mathbf{B}_{i+\frac{1}{2}})-\mathcal{P}(\mathbf{w^*}-\mathbf{B}_{i-\frac{1}{2}})(\mathbf{w^*}-\mathbf{B}_{i-\frac{1}{2}})\right]\\
        & =\dfrac{1}{\Delta x}\frac{g}{2}\left[\mathcal{P}\left(\mathbf{w^*}-\overline{\mathbf{B}}_{i}-\frac{1}{2}\Delta \mathbf{B}_i\right)\left(\mathbf{w^*}-\overline{\mathbf{B}}_{i}-\frac{1}{2}\Delta \mathbf{B}_i\right)\right.\\
        &\;\;\;\left.-\mathcal{P}\left(\mathbf{w^*}-\overline{\mathbf{B}}_{i}+\frac{1}{2}\Delta \mathbf{B}_i\right)\left(\mathbf{w^*}-\overline{\mathbf{B}}_{i}+\frac{1}{2}\Delta \mathbf{B}_i\right)\right]\\
        & = \dfrac{1}{\Delta x}\frac{g}{2}\left[\mathcal{P}(\mathbf{w^*}-\overline{\mathbf{B}}_i)\left(-\Delta \mathbf{B}_i\right)-\mathcal{P}\left(\frac{\Delta \mathbf{B}_i}{2}\right)\left(2\mathbf{w^*}-2\overline{\mathbf{B}}_i\right)\right]\\
        & = -g\mathcal{P}(\mathbf{w^*}-\overline{\mathbf{B}}_i)\left(\dfrac{\mathbf{B}_{i+\frac{1}{2}}-\mathbf{B}_{i-\frac{1}{2}}}{\Delta x}\right) = -g\mathcal{P}(\overline{\mathbf{h}}_i)\left(\dfrac{\mathbf{B}_{i+\frac{1}{2}}-\mathbf{B}_{i-\frac{1}{2}}}{\Delta x}\right).
    \end{aligned}
\end{equation}

\end{proof}
 In the meantime, \eqref{eq:balancedterm} reduces to the deterministic well-balanced quadrature approximation when there is no uncertainty. The deterministic formula is obtained by applying the midpoint quadrature rule to the cell averages \eqref{eq:semidiscretewsg} with the derivative term $\mathbf{B}_x(x_i)$ approximated by the finite difference $\left(\mathbf{B}_{i+\frac{1}{2}}-\mathbf{B}_{i-\frac{1}{2}}\right)/\Delta x$ \cite{kurganov2007second}.
\subsection{Hyperbolicity-Preserving CFL-type conditions}
To determine \\hyperbolicity-preserving CFL-type conditions, we focus on the first $K$ equations in \eqref{eq:semidiscretewsg} which prescribe evolution of $\overline{\mathbf{h}}_i$, 
\begin{equation}\label{eq:semidiscrete}
    \dfrac{d}{dt}\overline{\mathbf{h}}_i = -\dfrac{1}{\Delta x}\left[\mathcal{F}^{\hat{h}}_{i+\frac{1}{2}}(t)-\mathcal{F}^{\hat{h}}_{i-\frac{1}{2}}(t)\right],
\end{equation}
where 
\begin{equation}\label{eq:fluxcu}
    \mathcal{F}^{\hat{h}}_{i+\frac{1}{2}} = \dfrac{a^{+}_{i+\frac{1}{2}}\mathbf{q}_{i+\frac{1}{2}}^{-}-a^{-}_{i+\frac{1}{2}}\mathbf{q}_{i+\frac{1}{2}}^{+}}{a^{+}_{i+\frac{1}{2}}-a^{-}_{i+\frac{1}{2}}}+\dfrac{a^{+}_{i+\frac{1}{2}}a^{-}_{i+\frac{1}{2}}}{a^{+}_{i+\frac{1}{2}}-a^{-}_{i+\frac{1}{2}}}\left[\mathbf{h}_{i+\frac{1}{2}}^{+}-\mathbf{h}_{i+\frac{1}{2}}^{-}\right].
\end{equation}
A fully discrete version of \eqref{eq:semidiscrete} computes the unknowns at fixed values of time, $t^n$, $n \in \N_0$, with $t^n < t^{n+1}$. For example, with $\overline{\mathbf{h}}_i^n$ the numerical approximation to $\overline{\mathbf{h}}_i(t^n)$, and $\Delta t^n \coloneqq t^{n+1} - t^n$, the Forward Euler discretization of \eqref{eq:semidiscrete} reads,
\begin{align}\label{eq:FE}
  \overline{\mathbf{h}}_i^{n+1} &= \overline{\mathbf{h}}_i^n - \lambda_i^n \left[\mathcal{F}^{\hat{h}}_{i+\frac{1}{2}}(t^n)-\mathcal{F}^{\hat{h}}_{i-\frac{1}{2}}(t^n)\right], & \lambda_i^n \coloneqq \frac{\Delta t^n}{\Delta x_i}.
\end{align}
The following CFL condition guarantees hyperbolicity of the system \eqref{eq:FE} at $t = t^{n+1}$ for all cell averages, by enforcing the positivity condition prescribed in \cref{thm:h-positivity}.
\begin{lem}\label{lemma:CFL}
  Let $\left\{\xi_j\right\}_{j=1}^M$ be the nodes of a quadrature rule satisfying the conditions of \cref{thm:h-positivity}. Assume that $\overline{\mathbf{h}}_i^n(\xi_j)\;  >  \;0$ for $1 \leq j \leq M$. If $\Delta t^n$ satisfies
  \begin{align}\label{eq:CFL}
    \Delta t^n &< \Delta t^n_{h} \coloneqq 
    \min_{\substack{1\le j\le M\\i}}\left\{\Delta x_i \left|\dfrac{(\overline{\mathbf{h}}^{n}_{i})^{\mathrm{T}}\boldsymbol{\Phi}(\xi_j)}{\left[\mathcal{F}^{\hat{h}}_{i+\frac{1}{2}}(t_n)-\mathcal{F}^{\hat{h}}_{i-\frac{1}{2}}(t_n)\right]^{\mathrm{T}}\boldsymbol{\Phi}(\xi_j)}\right|\right\},
  \end{align}
  then the flux Jacobian \eqref{eq:jacobian1}, $J\left(\overline{\mathbf{U}}_i^{n+1}\right)$ is diagonalizable with real eigenvalues. 
\end{lem}
\begin{proof}
  \cref{thm:h-positivity} guarantees the conclusion if $\overline{\mathbf{h}}_i^{n+1}(\xi_j) > 0$, for $1 \leq j \leq M$, so we proceed to show this latter property. For each $j$, the inequality
  \begin{align}\label{ieq:positivity}
     0 < (\overline{\mathbf{h}}^{n+1}_{i})^{\mathrm{T}}\boldsymbol{\Phi}(\xi_j) = (\overline{\mathbf{h}}^{n}_{i})^{\mathrm{T}}\boldsymbol{\Phi}(\xi_j)-
    \lambda_i^n \left[\mathcal{F}^{\hat{h}}_{i+\frac{1}{2}}(t_n)-\mathcal{F}^{\hat{h}}_{i-\frac{1}{2}}(t_n)\right]^{\mathrm{T}}\boldsymbol{\Phi}(\xi_j)
  \end{align}
  holds if we choose 
  \begin{align*}
    \frac{\Delta t^n}{\Delta x_i} &= \lambda_i^n < \min_{1\le j\le M}\left\{\left|\dfrac{(\overline{\mathbf{h}}^{n}_{i})^{\mathrm{T}}\boldsymbol{\Phi}(\xi_j)}{\left[\mathcal{F}^{\hat{h}}_{i+\frac{1}{2}}(t_n)-\mathcal{F}^{\hat{h}}_{i-\frac{1}{2}}(t_n)\right]^{\mathrm{T}}\boldsymbol{\Phi}(\xi_j)}\right|\right\}. 
  \end{align*}
  Multiplying both sides by $\Delta x_i$ and minimizing over $i$ yields the conclusion.

\end{proof}

The condition \eqref{eq:CFL} ensures positivity of the water height, but we also need to adhere to standard wavespeed-based CFL stability conditions. Thus, we will choose
\begin{equation}\label{eq:tstep}
  \Delta t^n = 0.9\min\left\{\Delta t^n_h, \min_i \frac{\Delta x_i}{\max \{a^{+}_{i+\frac{1}{2}}, -a^{-}_{i+\frac{1}{2}}\}}\right\}.
\end{equation}
To extend these conditions to hold higher-order schemes, we use strong stability-preserving Runge-Kutta schemes \cite{gottlieb2001strong} to solve the semidiscrete system \eqref{eq:semidiscretewsg}. The analysis above for the condition \eqref{eq:CFL} still holds for this solver since the ODE solver can be written as a convex combination of several forward Euler steps. However, an adaptive time-step control needs to be adopted to determine the time step \cite{chertock2015well,kurganov2018finite}. The analysis above can also be naturally extended to any other finite volume solvers.
\begin{remark}
  The CFL condition \eqref{eq:CFL} can be relaxed if the signs of the fluxes are taken into account in the inequality \eqref{ieq:positivity}. In implementation, this can be used to reduce the simulation time.
\end{remark}

It is important to note that, the CFL-type condition provided above is limited to the cell averages. For the second-order (or higher-order) central-upwind scheme, additional correction is required for the pointwise reconstructions $\mathbf{U}^{\pm}_{i+\frac{1}{2}}$ to ensure hyperbolicity of \eqref{eq:FE}. Similarly, special correction is needed for the near-dry states, where the matrices $\mathcal{P}(\mathbf{h}^{\pm}_{i+\frac{1}{2}})$ are close to singular, to ensure hyperbolicity.
\subsubsection{Hyperbolicity-Preserving Correction to the Reconstruction}
Assuming $(\overline{\mathbf{h}}^{n}_{i})^{\mathrm{T}}\boldsymbol{\Phi}(\xi_j)>0$, we are able to enforce $(\overline{\mathbf{h}}^{n+1}_{i})^{\mathrm{T}}\boldsymbol{\Phi}(\xi_j)>0$ for $j = 1, \cdots,M$ under the CFL-type condition \eqref{eq:tstep}, see \cref{lemma:CFL}. However, the one-sided propagation speeds \eqref{eq:pspeed} in the central-upwind scheme \eqref{eq:FE} are estimated by the eigenvalues of the Jacobian $\frac{\partial F}{\partial \hat{U}}$ using the pointwise values at the cell interfaces. Thus, computation of these wave speeds requires positivity of the pointwise reconstruction at quadrature points, i.e., $(\mathbf{h}^{\pm}_{i+\frac{1}{2}})^{T}\boldsymbol{\Phi}(\xi_j)>0$, which is not guaranteed by $(\overline{\mathbf{h}}^{n}_{i})^{\mathrm{T}}\boldsymbol{\Phi}(\xi_j)>0$. To resolve this problem, we use the filtering strategy proposed in \cite{schlachter2018hyperbolicity} to filter $\mathbf{h}_{i+\frac{1}{2}}^{\pm}$. 

Given a polynomial $p_{\hat{y}}(\xi) = \sum_{k=1}^{K}\hat{y}_k\phi_k(\xi)$ with positive moment $\hat{y}_1$, we find the smallest possible weight $\mu'$ such that the weighted averages of the polynomial $p_{\hat{y}}(\xi)$ and the moment $\hat{y}_1$ are nonnegative at given quadrature points $\{\xi_j\}_{j = 1}^M$, i.e.,
\begin{equation}\label{eq:filtering}
    \mu' \hat{y}_1+(1-\mu')p_{\hat{y}}(\xi)\ge 0\Leftrightarrow \hat{y}_1+\sum_{k=2}^{K}(1-\mu')\hat{y}_k\phi_k(\xi_j)\ge 0, j = 1,\cdots, M,
\end{equation}
and the coefficients of the polynomial are filtered by
\begin{align}\label{eq:filterPCE}
    &\hat{\mathsf{y}}_1 = \hat{y}_1, &\hat{\mathsf{y}}_k = (1-\mu)\hat{y}_k, k = 2,\cdots, K,
\end{align}
where $\mu = \min\{\mu'+\delta, 1\}$, and we select $\delta = 10^{-10}$ in our scheme. Hence, the filtered polynomial $p_{\hat{\mathsf{y}}}(\xi) = \sum_{k=1}^{K}\hat{\mathsf{y}}_k\phi(\xi)$ is positive at given quadrature points $\{\xi_j\}_{j = 1}^M$. We filter $p_{\hat{y}}(\xi) = \sum_{k=1}^{K}\hat{y}_k\phi_k(\xi)$ and $p_{\hat{z}}(\xi) = \sum_{k=1}^{K}\hat{z}_k\phi_k(\xi)$ simultaneously by calculating the individual filtering parameters $\mu'_{\hat{y}}$ and $\mu'_{\hat{z}}$ for $p_{\hat{y}}(\xi)$ and $p_{\hat{z}}(\xi)$, respectively, through \eqref{eq:filtering}. Then the simultaneous filtering parameter is set to $\mu = \min\{\mu'_{\hat{y}}+\delta, \mu'_{\hat{z}}+\delta, 1\}$.

We will exercise the filtering strategy \eqref{eq:filtering}-\eqref{eq:filterPCE} for pointwise reconstructions.
We compute the filtering parameter $\mu^n_i$ at time $t = t^n$ for the $i$th cell for  $(\mathbf{h}^{\pm}_{i\mp\frac{1}{2}})^{\mathrm{T}}\boldsymbol{\Phi}(\xi)$ according to \eqref{eq:filtering}. The pointwise reconstructions $\mathbf{h}^{\pm}_{i\mp\frac{1}{2}}$ are then filtered by
\begin{equation}\label{eq:filterh}
\begin{aligned}
  &\left(\mathsf{h}^{\pm}_{i\mp\frac{1}{2}}\right)_1 = \left(\mathbf{h}^{\pm}_{i\mp\frac{1}{2}}\right)_1,
  &\left(\mathbf{h}^{\pm}_{i\mp\frac{1}{2}}\right)_k = (1-\mu^n_i)\left(\mathsf{h}^{\pm}_{i\mp\frac{1}{2}}\right)_k,k=2,\cdots, K.
\end{aligned}
\end{equation}
The corresponding cell average is adjusted accordingly in order to remain consistent,
\begin{equation}\label{eq:filterhavg}
    \overline{\mathsf{h}}^n_i = \frac{1}{2}\left(\mathsf{h}^{+}_{i-\frac{1}{2}}+\mathsf{h}^{-}_{i+\frac{1}{2}}\right).
\end{equation}
\begin{remark}
  To reduce oscillations in $q_\Lambda(x,t,\xi)$, we can also filter the discharge reconstructions $\mathbf{q}^{\pm}_{i-\frac{1}{2}}$.
The corresponding cell average needs to be adjusted similarly to \eqref{eq:filterhavg}. In \cref{ssec:results-sdb} when $(\alpha, \beta) = (1,3)$, we adopt this filtering approach to reduce oscillations in the discharge.
\end{remark}
As an alternative to the filtering above, one can use a convex-optimization based method \cite{boyd2004convex} to enforce the positivity of $(\mathbf{h}^{\pm}_{i\mp\frac{1}{2}})^{\mathrm{T}}\boldsymbol{\Phi}(\xi)$ at quadrature points $\{\xi_j\}_{j=1}^M$.
\subsubsection{Near-Dry State Correction}
When the polynomial $(\overline{\mathbf{h}}^{n}_{i})^{\mathrm{T}}\boldsymbol{\Phi}(\xi)\sim 0$, two issues related to the dry state may occur. One is that the first moments of the polynomials $(\mathbf{h}^{\pm}_{i\mp\frac{1}{2}})^{\mathrm{T}}\boldsymbol{\Phi}(\xi)$ may become nonpositive. This can happen even when the system is deterministic \cite{kurganov2007second}. Nonpositive first moments may lead to the failure of the filtering correction \eqref{eq:filtering}-\eqref{eq:filterPCE}. In our scheme, we adopt the following correction for nonpositive first moments. Denote the first moments of $\mathbf{h}^{\pm}_{i\mp\frac{1}{2}}$ by  $\left(\mathbf{h}^{\pm}_{i\mp\frac{1}{2}}\right)_1$, then
\begin{align}
  \textrm{if } \left(\mathbf{h}^{\pm}_{i\mp\frac{1}{2}}\right)_1\le0 \hskip 5pt 
  \textrm{then take } \mathbf{h}^{\pm}_{i\mp\frac{1}{2}}=\mathbf{0},\;\;
                      \mathbf{h}^{\mp}_{i\pm\frac{1}{2}}=2\overline{\mathbf{h}}^n_i.
\end{align}
Note that, this strategy reduces to a similar correction in the central-upwind scheme for the deterministic shallow water equations \cite{kurganov2007second}.

Another issue may happen when the matrix $\mathcal{P}(\mathbf{h}^{+}_{i+\frac{1}{2}})$ or $\mathcal{P}(\mathbf{h}^{-}_{i+\frac{1}{2}})$ is ill-conditioned, which may lead to problems with round-off errors when solving the corresponding linear system \eqref{eq:uPCE}. To resolve this issue, we extend to the stochastic model the desingularization process for the deterministic problem \cite{kurganov2007second,kurganov2018finite}. We demonstrate our correction using the matrix $\mathcal{P}(\mathbf{h}^{-}_{i+\frac{1}{2}})$ as an example. Let 
$$
\mathcal{P}(\mathbf{h}^{-}_{i+\frac{1}{2}}) = Q^{\mathrm{T}}\Pi Q,
$$
be the eigenvalue decomposition for $\mathcal{P}(\mathbf{h}^{-}_{i+\frac{1}{2}})$, where $\Pi = \text{diag}(\lambda_1,\cdots,\lambda_K)$. For $k = 1, \ldots, K$ and a given $\epsilon > 0$, define
\begin{align}\label{eq:invcorrected}
  \Pi^{\co} &= \text{diag}(\lambda^{\co}_1,\cdots,\lambda^{\co}_K),&
    \lambda^{\co}_k = \frac{\sqrt{2}\lambda_{k}}{\sqrt{\lambda^4_{k}+\max\{\lambda^4_{k},\epsilon^4\}}}.
\end{align}
In our scheme we choose $\epsilon = \Delta x$.
Then, the corrected PCE coefficient vector for the velocity $\mathbf{u}^{-}_{i+\frac{1}{2}}$ is given by 
\begin{equation}\label{eq:ucorrected}
    \mathbf{u}^{-}_{i+\frac{1}{2}} = Q^{\mathrm{T}}\Pi^{\co} Q\mathbf{q}^{-}_{i+\frac{1}{2}}.
\end{equation}
For well-conditioned $\mathcal{P}(\mathbf{h}^{-}_{i+\frac{1}{2}})$, the correction \eqref{eq:ucorrected} reduces to the system \eqref{eq:uPCE}, but when $\mathcal{P}(\mathbf{h}^{-}_{i+\frac{1}{2}})$ is near singular, the discharge needs to be recomputed, 
\begin{equation}\label{eq:qcorrected}
    \mathbf{q}^{-}_{i+\frac{1}{2}} = \mathcal{P}(\mathbf{h}^{-}_{i+\frac{1}{2}})\mathbf{u}^{-}_{i+\frac{1}{2}},
\end{equation}
in order to keep the scheme consistent.
\begin{remark}
If there is no uncertainty, the correction \eqref{eq:invcorrected}-\eqref{eq:ucorrected} reduces to the deterministic velocity desingularization in \cite{kurganov2007second,kurganov2018finite}.
\end{remark}

\section{Numerical Results}\label{sec:results}
In this section, we summarize numerical tests to illustrate robustness of the proposed schemes for the SGSWE system \eqref{eq:swesg5} with different uncertainty models and parametric distributions. For simplicity we consider only one-dimensional stochastic spaces ($d=1$) associated to a Beta density over $[-1,1]$,
\begin{align*}
  \rho(\xi) \coloneqq \rho^{(\alpha,\beta)}(\xi) &= C(\alpha,\beta) (1-\xi)^\alpha (1+\xi)^\beta, & 
  C(\alpha,\beta)^{-1} &= 2^{\alpha+\beta+1} B(\beta+1, \alpha+1)
\end{align*}
where $B(\cdot,\cdot)$ is the Beta function, and the parameters $\alpha, \beta > -1$ can be chosen freely and control how mass concentrates at $\xi = 1$ and $\xi = - 1$, respectively. In particular $\alpha = \beta = 0$ corresponds to the uniform distribution on $[-1,1]$. The numerical examples in the coming sections consist of the following numerical experiments:
\begin{itemize}[leftmargin=*]
  \item \cref{ssec:results-sbt}: Stochastic bottom topography model, comparing the SGSWE solution \eqref{eq:swesg5} with $K=9$ and $K=17$ with the uniform density, $\alpha = \beta = 0$. The results are compared against a $K=9$ stochastic collocation solution computed with $S = 100$ stochastic points. The stochastic collocation solution for, e.g., the water height $h$, is computed via quadrature,
    {\begin{align*}
      h_{SC}(x,t,\xi) &\coloneqq \sum_{j=1}^K \hat{h}_{SC,j}(x,t) \phi_k(\xi), & \hat{h}_{SC,j}(x,t) &\coloneqq \sum_{s=1}^S h(x,t,\zeta_s) \phi_j(\zeta_s) z_s
    \end{align*}}
    where $\{\zeta_s, z_s\}_{s=1}^S$ is the $S$-point $\rho$-Gaussian quadrature rule, and $h(x,t,\zeta_s)$ is a numerical solution to a deterministic specialization of the SWE \eqref{eq:swesg1} obtained by setting $\xi = \zeta_s$ and numerically solved using a deterministic central-upwind scheme.
  \item \cref{ssec:results-sws}: Stochastic water surface model, testing the well-balanced property of the scheme with $\alpha = \beta = 0$
  \item \cref{ssec:results-sdb}: Stochastic discontinuous bottom topography model, investigating the effects of different values of $M$ used to enforce $\mathcal{P}(\hat{h}) > 0$. This example also investigates different distributions, with $(\alpha,\beta) = (3,1)$ and $(\alpha,\beta) = (1,3)$. 
\end{itemize}

%


 The parameter $\theta$ in the generalized minmod limiter is set to $\theta = 1.3$ for the first two examples, and $\theta=1$ for the third example. The gravitational constant $g$ is set to $g=1$ for the first two examples, and $g=2$ for the last example. We filter only the water heights $h_\Lambda$ except in the very last numerical test. In the third numerical example, when $(\alpha,\beta) = (1,3)$, we filter both the water heights and the discharges of the water. In all examples, the CFL condition we use in our simulation is \eqref{eq:tstep}. However, we observe that in practice, a relaxed time step $c\Delta t^n (c>1)$ will not result in loss of hyperbolicity and the plots are similar visually to the results obtained from the condition \eqref{eq:tstep}. We believe this is because condition \eqref{eq:h-positivity} is only a sufficient but not a necessary condition to the hyperbolicity of SGSWE.

Our numerical results will report quantile regions indicating the range of behavior for solutions. These quantile regions are computed empirically by computing the corresponding PCEs on $10^5$ randomly sampled points from the density $\rho$ on $[-1,1]$. 

For a fixed spatial grid, the computational cost depends on the dimension $K$ of the chosen polynomial subspace $P_{\Lambda}$. In order to compute the propagation speeds \eqref{eq:pspeed}, the eigenvalues of the $2 K \times 2 K$ Jacobian $J(\mathbf{U})$ matrix must be computed, making this cost increase as $K$ increases. In addition, to preserve hyperbolicity, we need to ensure the positivity of the water height at all the quadrature points for every spatial-temporal point (\cref{thm:hyperbolicity}). Therefore, the cost for preserving the hyperbolicity is at most of order $O(K^3)$ per cell per time step (\cref{cor:tchakaloff}). These relations are formally independent of the dimension $d$ of the stochastic space, but in practice $K$ can grow considerably as $d$ is increased. For example, one may choose $P_{\Lambda}$ to be the space of the polynomials with degree up to $L$. In this case, $K = \binom{L+d}{d}$. When $L\ge d$, as $d$ increases, $K$ increases and also therefore does the computational cost. In this paper, we only consider numerically the case $d=1$. We plan to investigate higher dimensional stochastic space in a future work. However, note that the developed theory in \cref{sec:model} and \cref{sec:hyperbolicity} extends to $d>1$.
\subsection{Stochastic Bottom Topography}\label{ssec:results-sbt}
We consider the shallow water system with deterministic initial conditions
\begin{equation}\label{eq:IV1}
w(x,0) = \left\{\begin{aligned}&1&&x<0\\ &0.5&&x>0\end{aligned}\right.,\quad q(x,0) = 0,
\end{equation} 
and with a stochastic bottom topography
\begin{equation}\label{eq:bottom 1}
B(x,\xi) = \left\{\begin{aligned}0.125(\cos(5\pi x)+2)+0.125\xi,\quad&|x|<0.2\\0.125+0.125\xi,\quad&\text{otherwise}\end{aligned}\right..
\end{equation}

     In this example, we model $\xi$ as a uniform random variable ($\alpha=\beta=0$). The corresponding orthonormal basis functions $\phi_j$ are the orthonormal Legendre polynomials on $[-1,1]$ with density $\rho(\xi)=\frac{1}{2}$. Initially, the highest possible bottom barely touches the initial water surface at $x=0.5$. In \cref{fig:ex1-w} and \cref{fig:ex1-q}, we use a uniform grid size $\Delta x$ over the physical domain $x \in [-1,1]$, and compute up to terminal time $t = 0.8$.  We present the numerical solutions for $K=9$ and $K=17$ using $M = 17$ and $M=33$-point Gaussian quadrature nodes, respectively, to enforce the positivity condition \eqref{eq:h-positivity}. 
\begin{figure}[htbp]
    \centering
    \includegraphics[width = .49\textwidth]{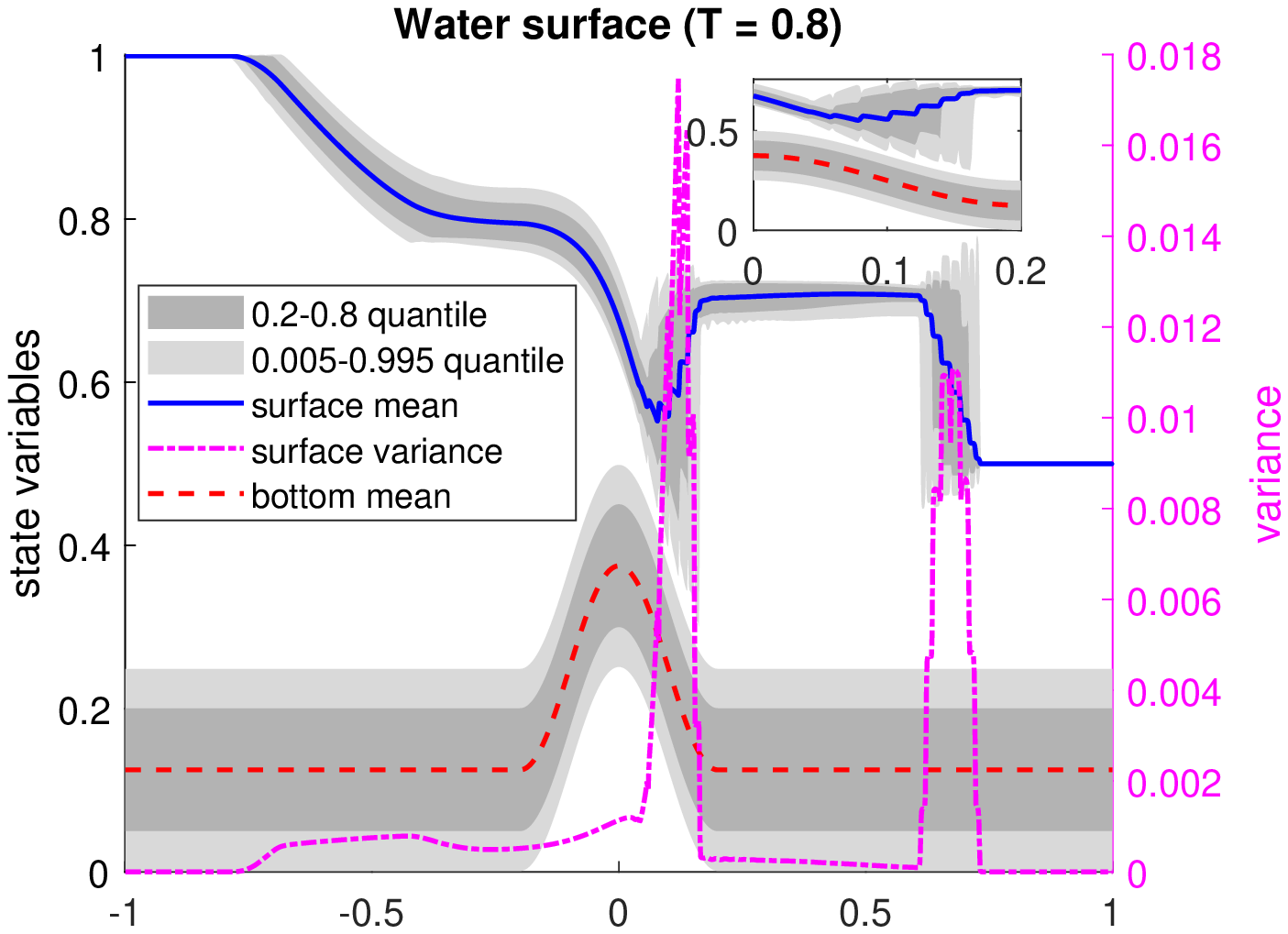}
    \hfill
    \includegraphics[width = .49\textwidth]{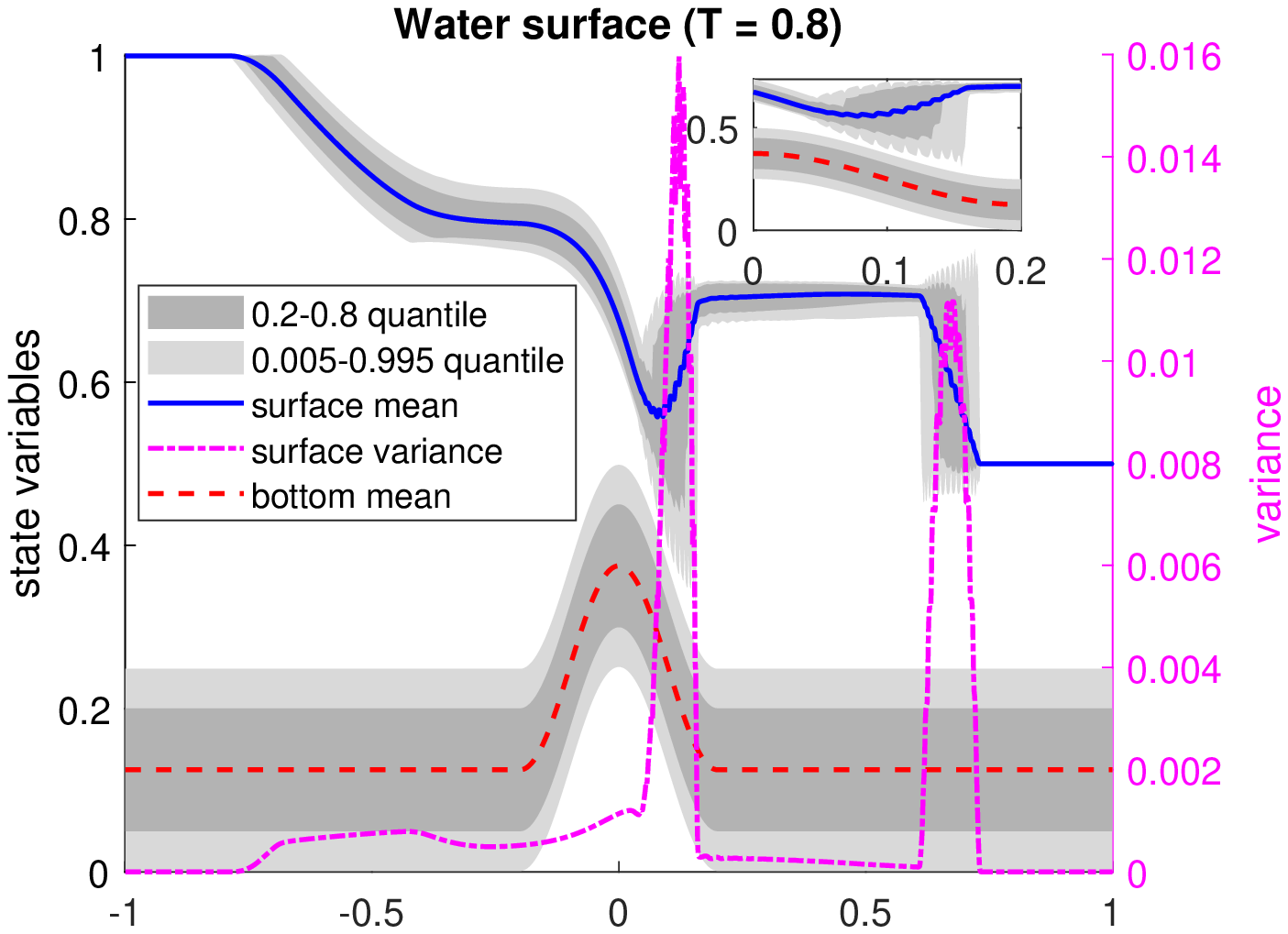}
    \hfill    
    \includegraphics[width = .49\textwidth]{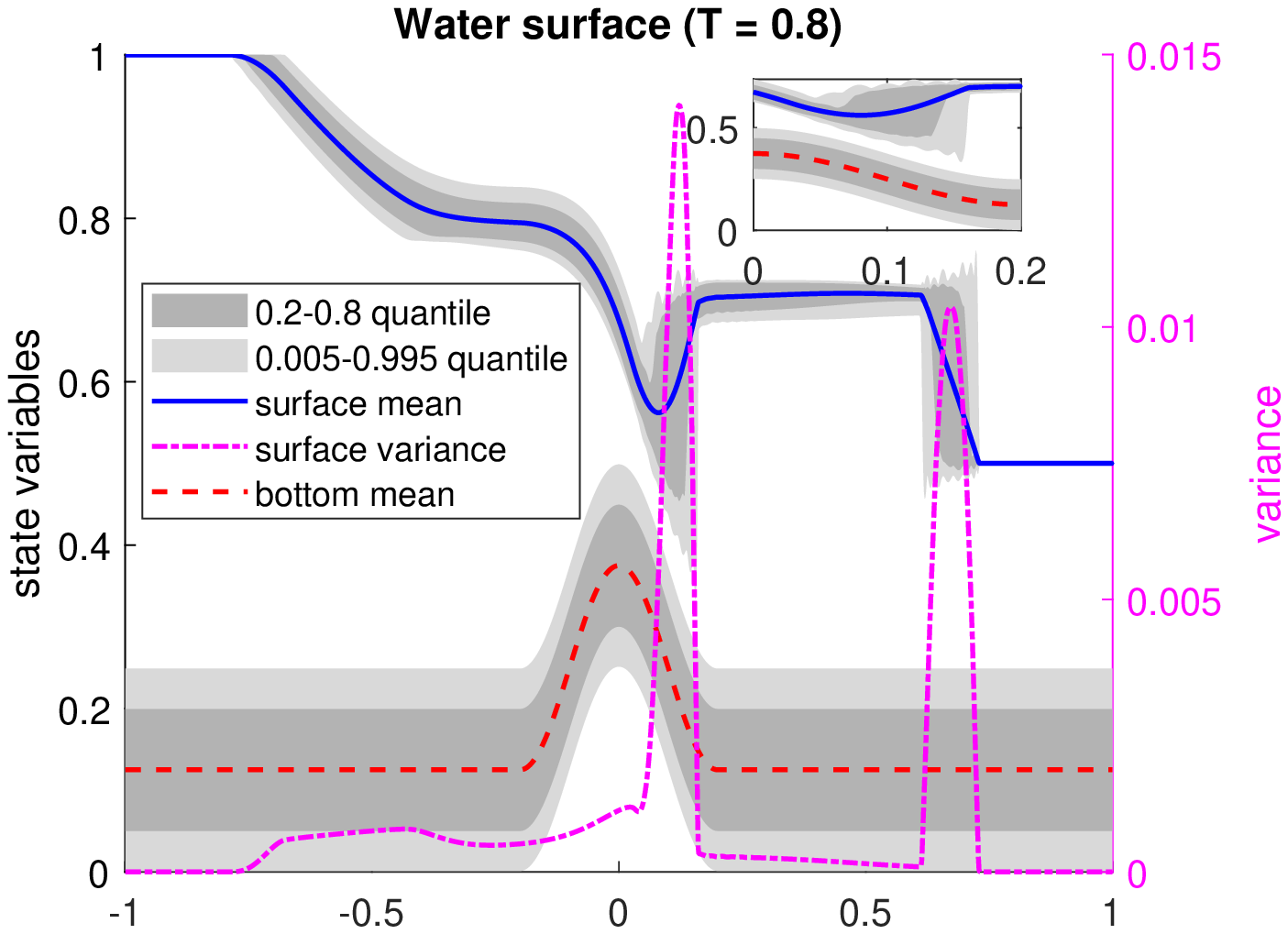}    
    \caption{Results for \cref{ssec:results-sbt}, water surfaces. Top left: stochastic Galerkin, $K=9, \Delta x = 1/800$. Top right: stochastic Galerkin, $K=17, \Delta x = 1/800$. Bottom: stochastic collocation, $K=9, \Delta x = 1/800$.}
    \label{fig:ex1-w}
\end{figure}
\begin{figure}[htbp]
  \centering
  \includegraphics[width = .49\textwidth]{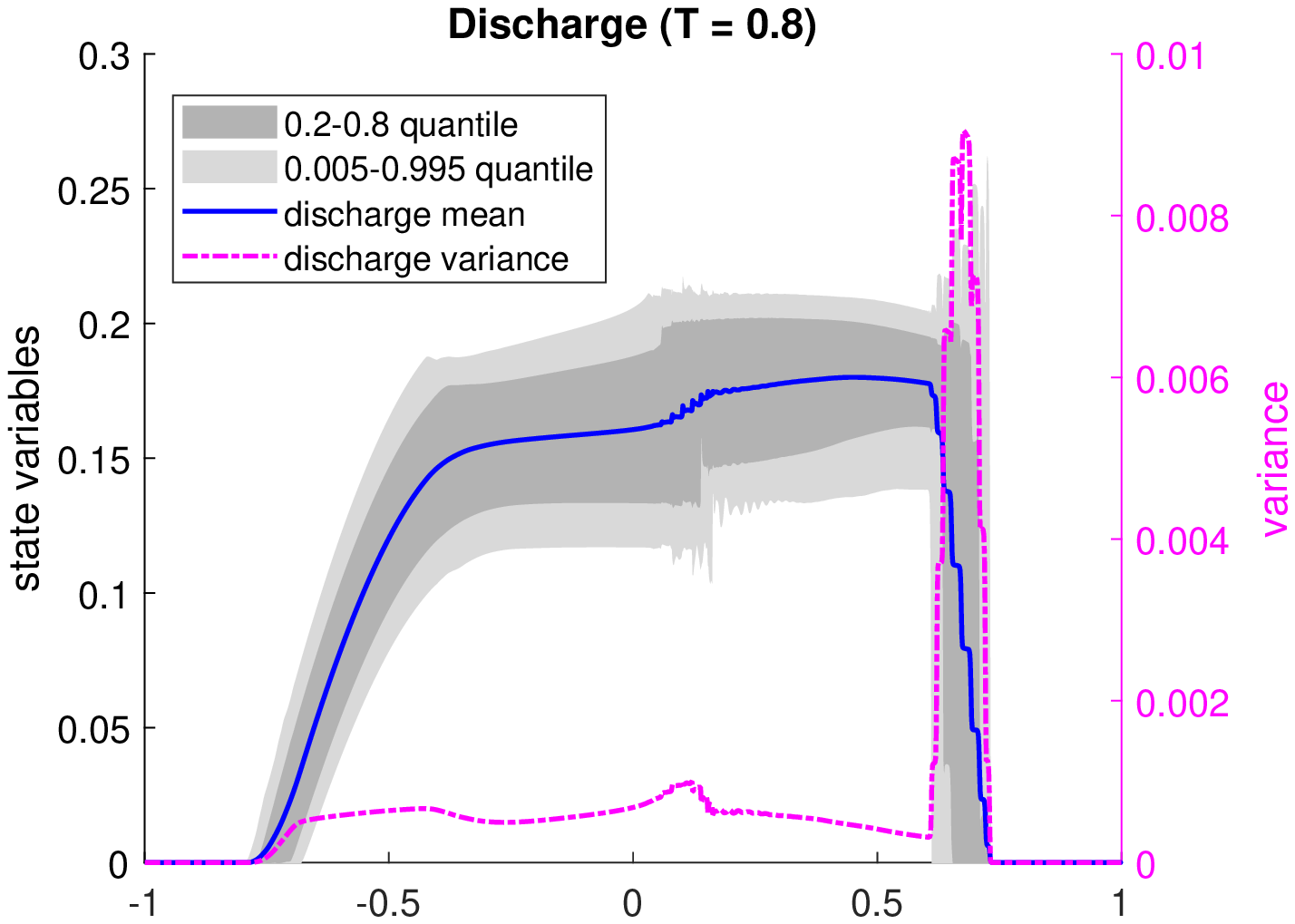}
  \hfill
  \includegraphics[width = .49\textwidth]{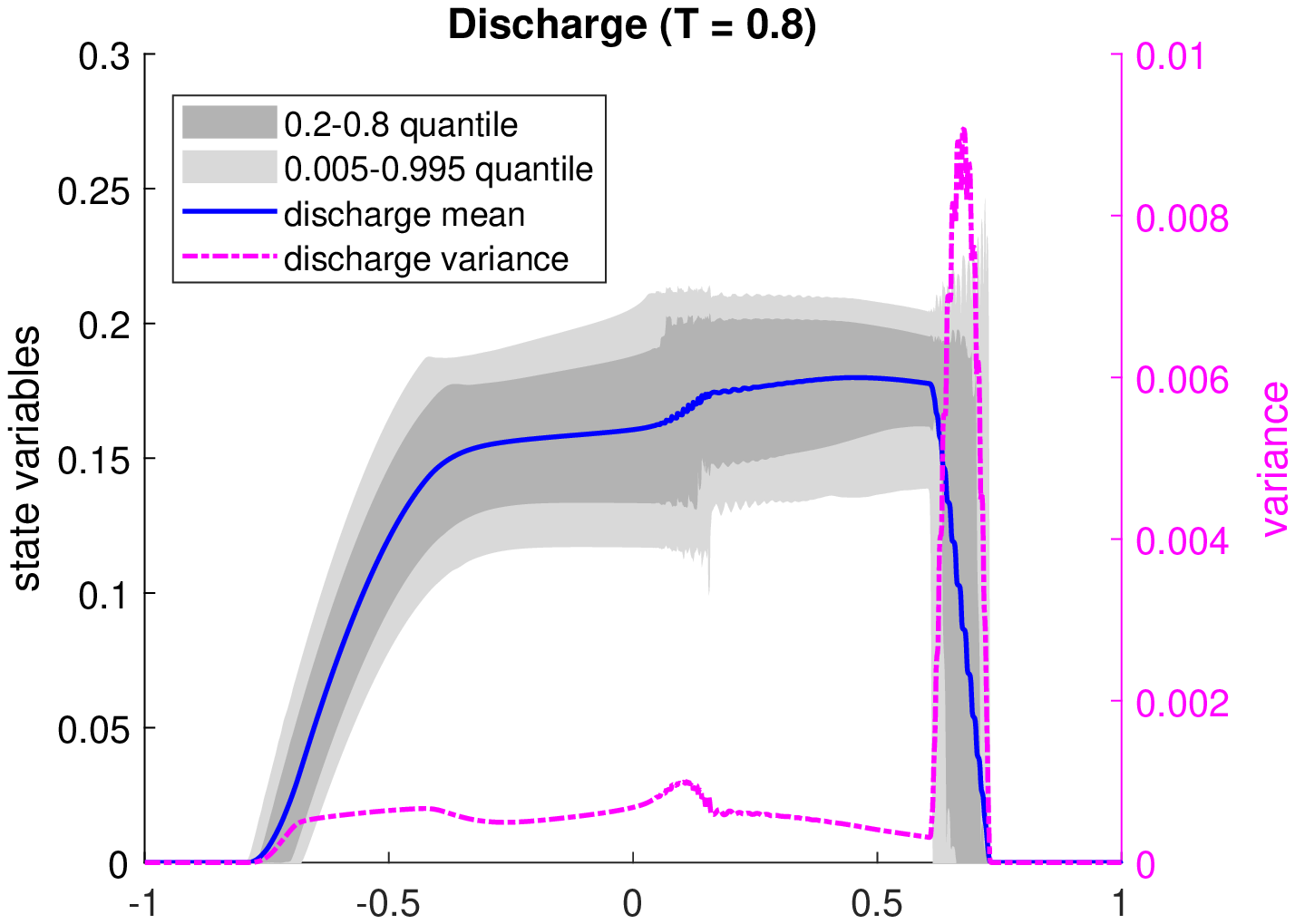}   
  \hfill    
  \includegraphics[width = .49\textwidth]{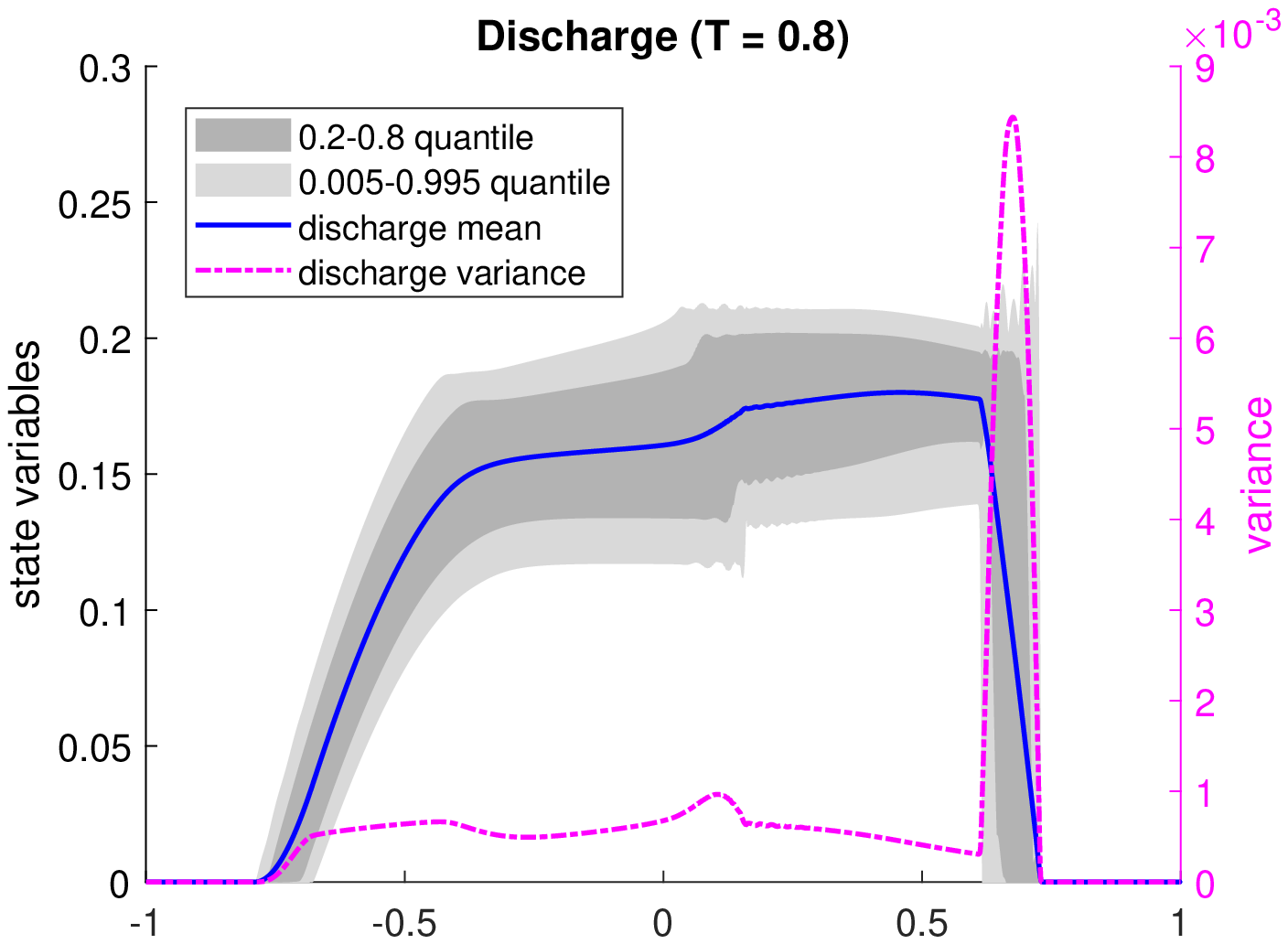}      
  \caption{Results for \cref{ssec:results-sbt}, discharges. Top left: stochastic Galerkin, $K=9, \Delta x = 1/800$. Top right: stochastic Galerkin, $K=17, \Delta x = 1/800$. Bottom: stochastic collocation, $K=9, \Delta x = 1/800$.}
  \label{fig:ex1-q}
\end{figure}

The $99\%$ confidence region of the water surface stays above the $99\%$ confidence region of the bottom function in the first three (top left, top right, bottom left) subfigures in \cref{fig:ex1-w}. 

For reference and comparison, a solution obtained by the stochastic collocation method ($100$ quadrature points, $K=9$-term PCE as explained in \cref{sec:results}) is computed. Results for water surface and discharge are shown in the right subfigures of \cref{fig:ex1-w} and \cref{fig:ex1-q}, respectively. We note that the stochastic collocation solution is a different PDE model, so we do not necessarily expect the numerical results from the SG and SC solvers to be identical for a fixed, finite $K$. In particular, we do not expect ``convergence" of one model to the other as, say $S \uparrow \infty$ and/or $\Delta x \downarrow 0$. However, the results in the figures do show substantial similarity between these solutions. The numerical solution obtained from the collocation method is less oscillatory near sharp gradients of water surface and discharges.

We observe small oscillations near sharp gradients of the water surface and discharge in all of the figures. We investigate the oscillations for the discharge more carefully in \cref{fig:ex1-q-oscillations}. We observe that both higher resolution and larger $K$ can reduce the magnitude of the oscillations that appear in quantiles.
\begin{figure}[htbp]
  \centering
  \hfill
  \includegraphics[width = .32\textwidth]{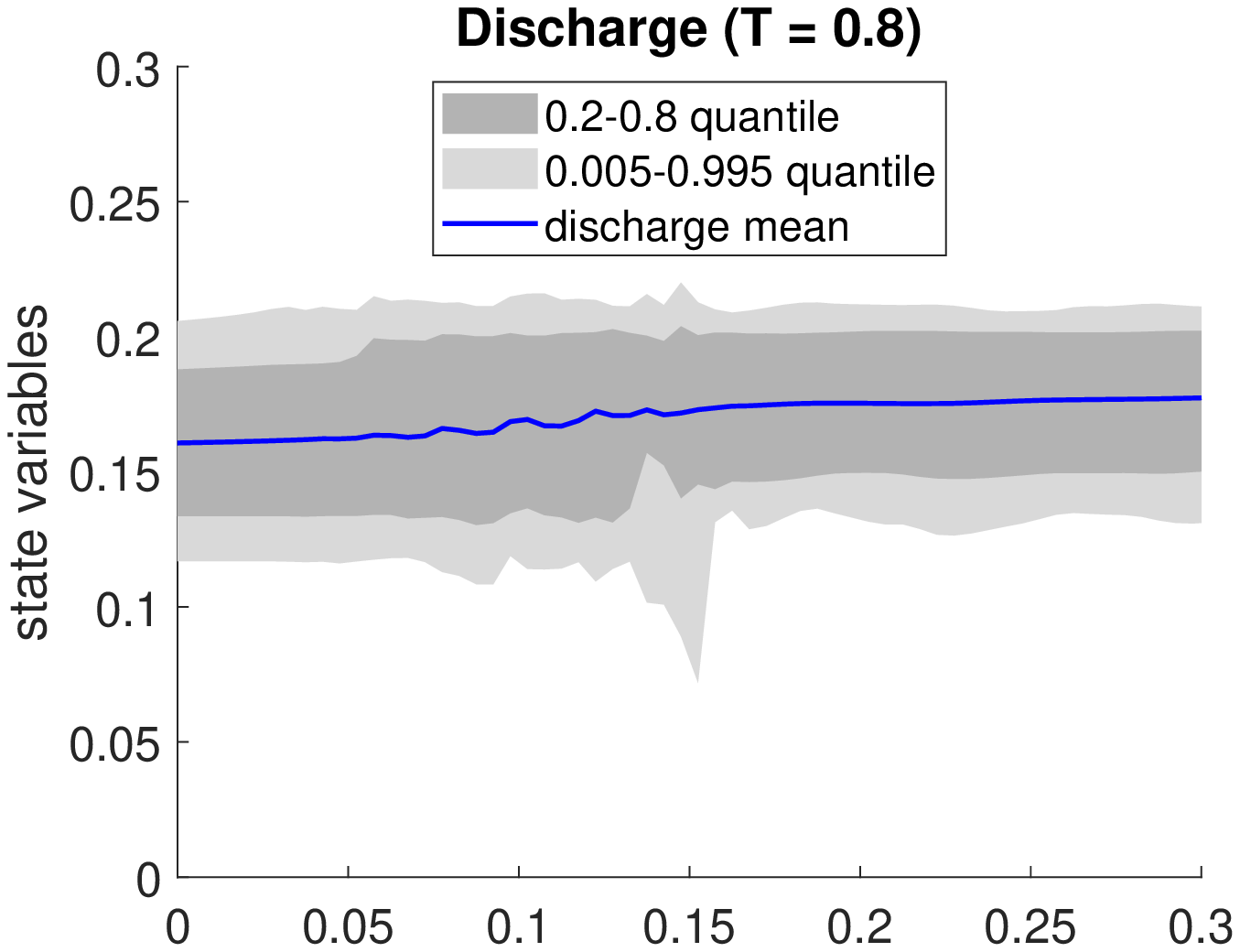}
  \hfill
  \includegraphics[width = .32\textwidth]{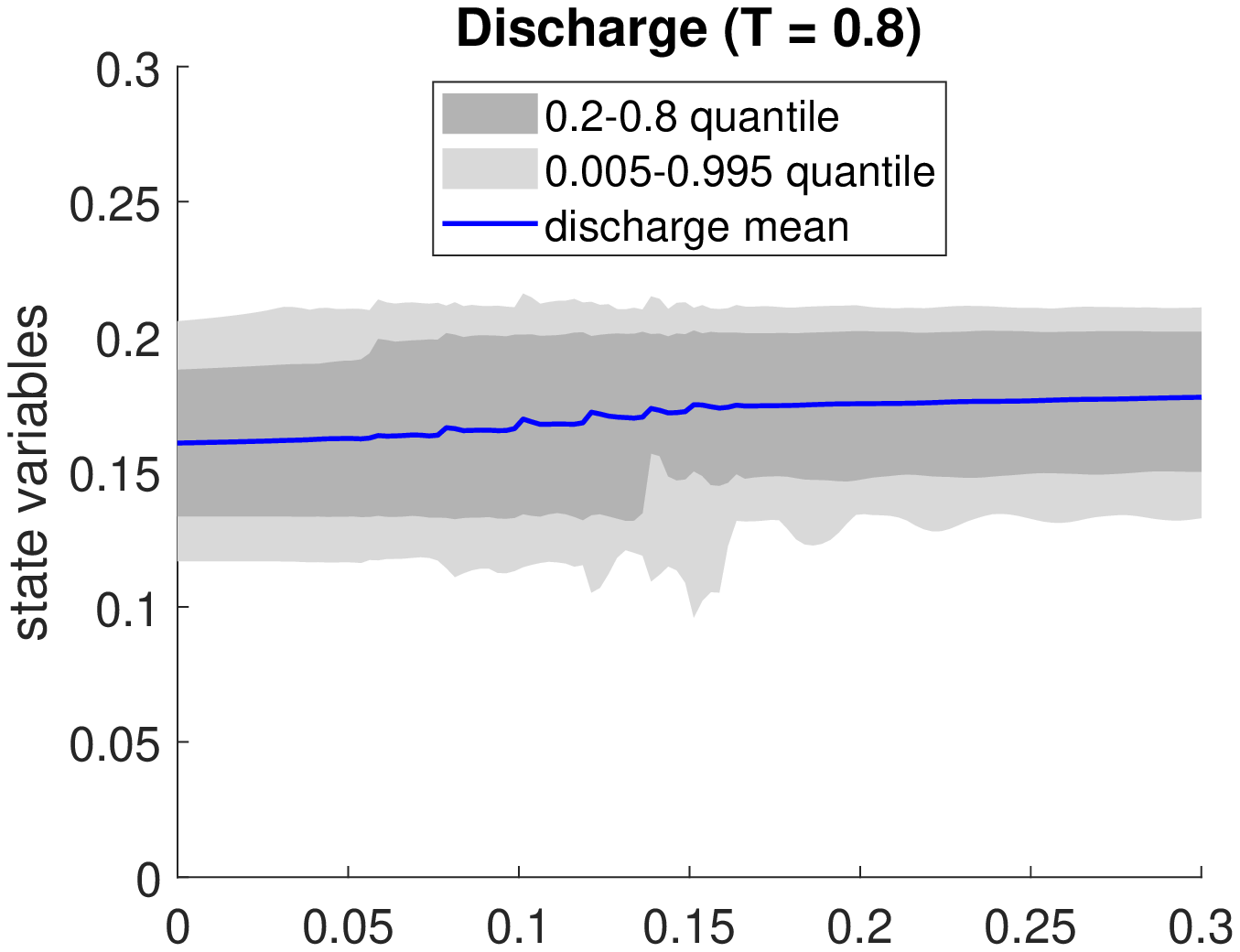} 
  \hfill    
  \includegraphics[width = .32\textwidth]{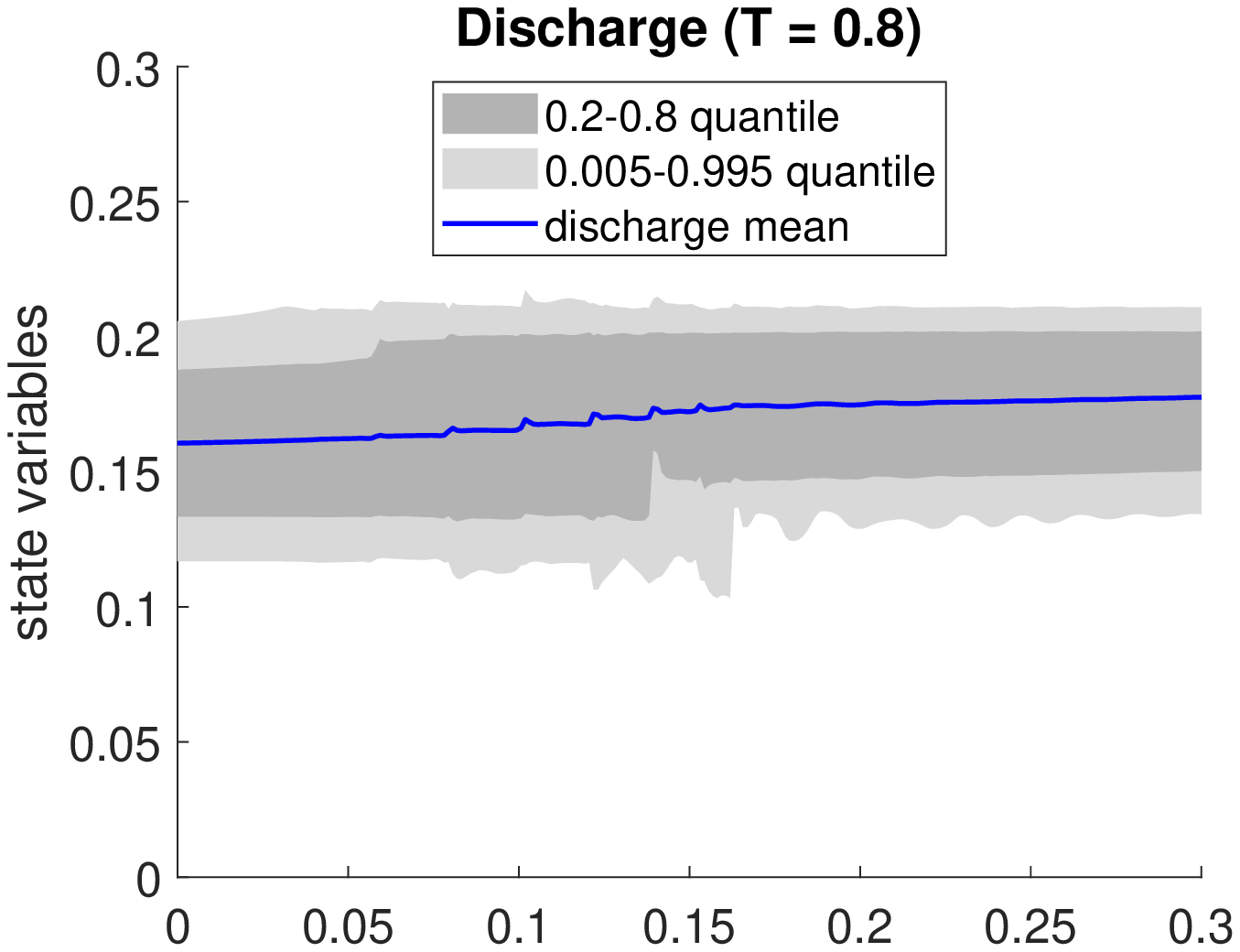}
  \hfill
  \includegraphics[width = .32\textwidth]{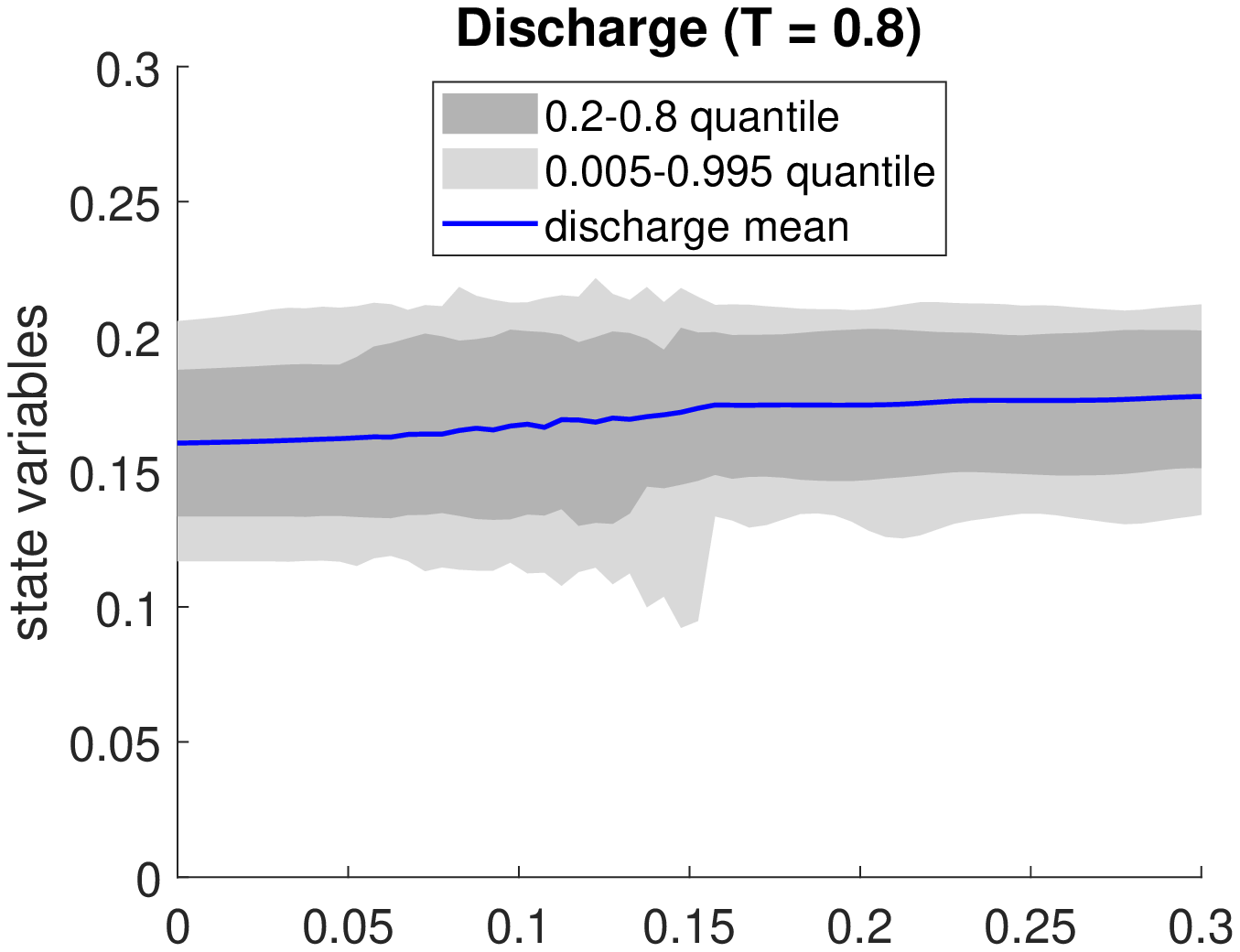}
  \hfill
  \includegraphics[width = .32\textwidth]{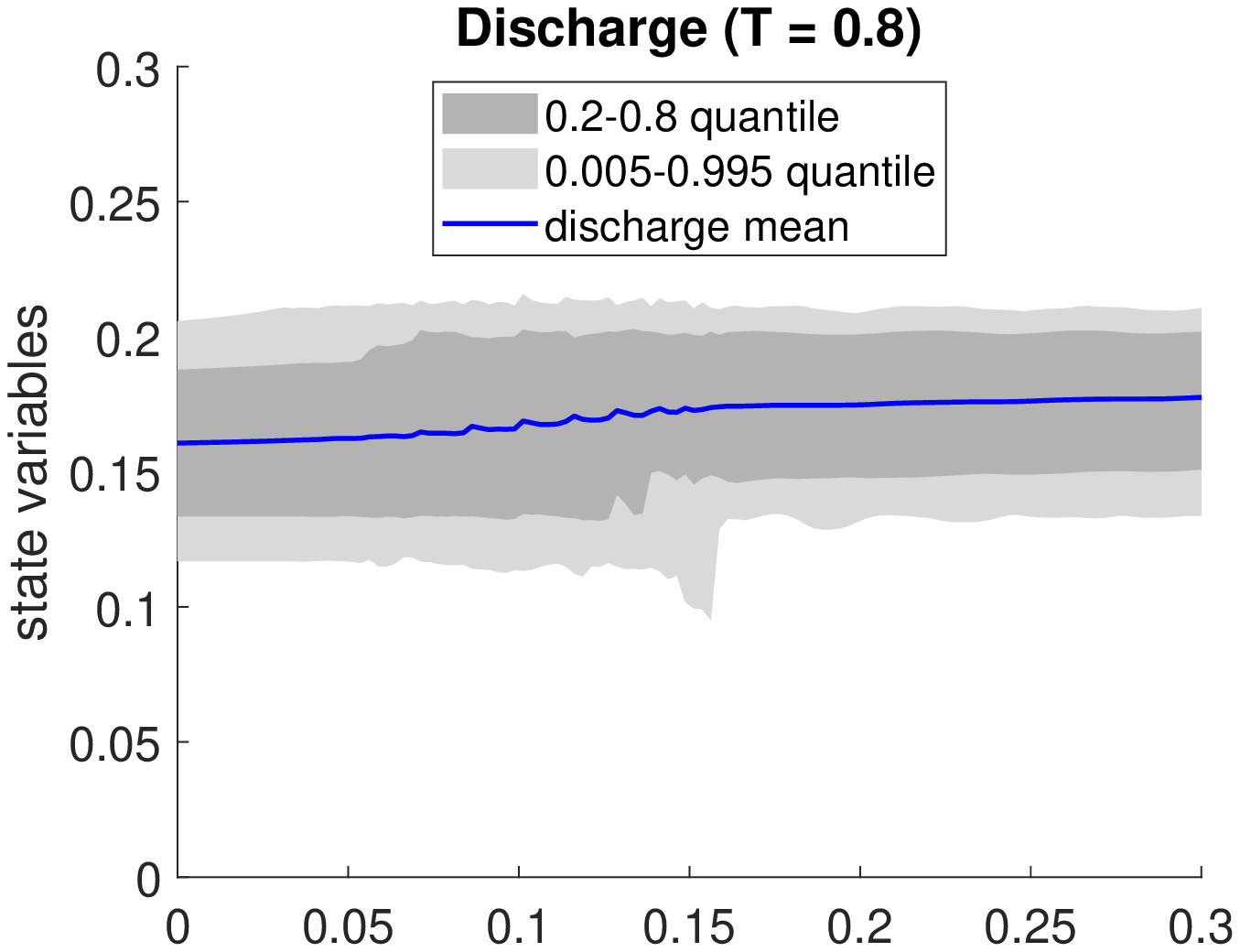}
  \hfill    
  \includegraphics[width = .32\textwidth]{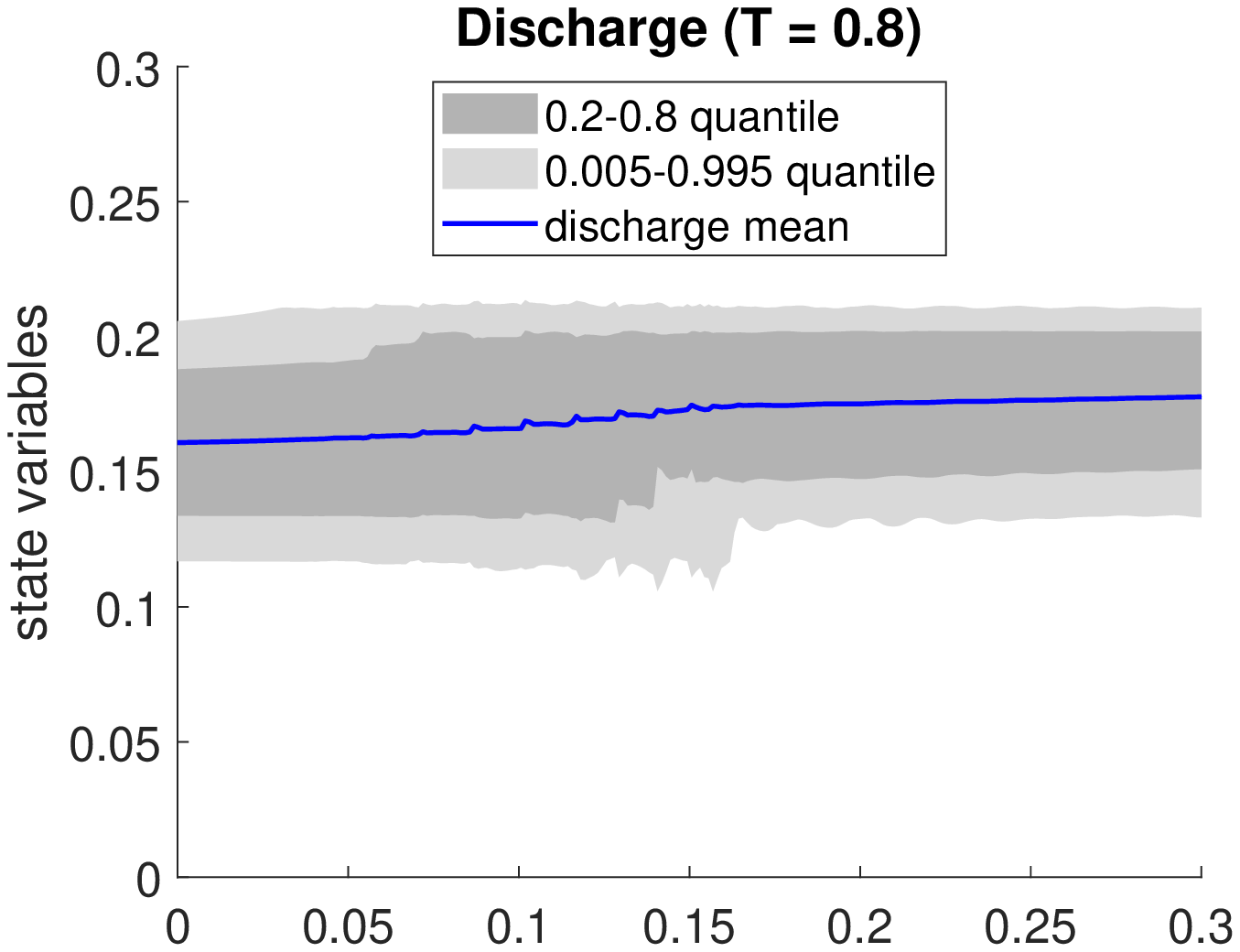}
  \hfill
  \caption{Results for \cref{ssec:results-sbt}, discharges on $[0, 0.3]$ for different values of $K$ and $\Delta x$, zoom view. Top: $K=9$;  bottom: $K = 13$. Left: $\Delta x = 1/200$; middle: $\Delta x = 1/400$; right $\Delta x = 1/800$.}
  \label{fig:ex1-q-oscillations}
\end{figure}

\subsection{Stochastic Water Surface}\label{ssec:results-sws}
Consider a stochastic shallow water system with a deterministic bottom function 
\begin{equation}\label{eq:bottom 2}
B(x,\xi) = \left\{\begin{aligned}
&10(x-0.3),&&0.3\le x\le 0.4,\\
&1-0.0025\sin^2(25(\pi(x-0.4))),&&0.4\le x\le 0.6,\\
&-10(x-0.7),&&0.6\le x\le 0.7,\\
&0&&\text{otherwise},\end{aligned}\right.
\end{equation}
and a stochastic water surface,
\begin{equation}\label{eq:IV2}
w(x,0,\xi) = \left\{\begin{aligned}&1.001+0.001\xi&&0.1<x<0.2,\\ &1&&\text{otherwise},\end{aligned}\right.\qquad q(x,0,\xi) \equiv 0. 
\end{equation}
We again model $\xi$ as a uniform random variable ($\alpha = \beta = 0$) with $K = 9$. 
A small uncertain region was originally at $0.1\le x\le 0.2$, where the water surface is slightly perturbed. 
The $17$-point $\rho$-Gaussian quadrature rule is used to enforce the condition \eqref{eq:h-positivity} to guarantee hyperbolicity. We compute the cell averages of the vector of PCE coefficients for water surface and discharges at terminal time $t=1.0$ on the physical domain $[-1,1]$ with uniform grid size $\Delta x = 1/400$. We observe from the mid figure of \cref{fig:ex2} that the perturbed water surface with uncertainties propagate along different directions. The right-moving wave interacts with the nonflat bottom and get partially reflected. The magnitude of the uncertainties doesn't seem to exceed the magnitude of the initial uncertainties, which illustrate the well-balanced property of our scheme. 

\begin{figure}[h]
    \centering
    \includegraphics[width = .49\textwidth]{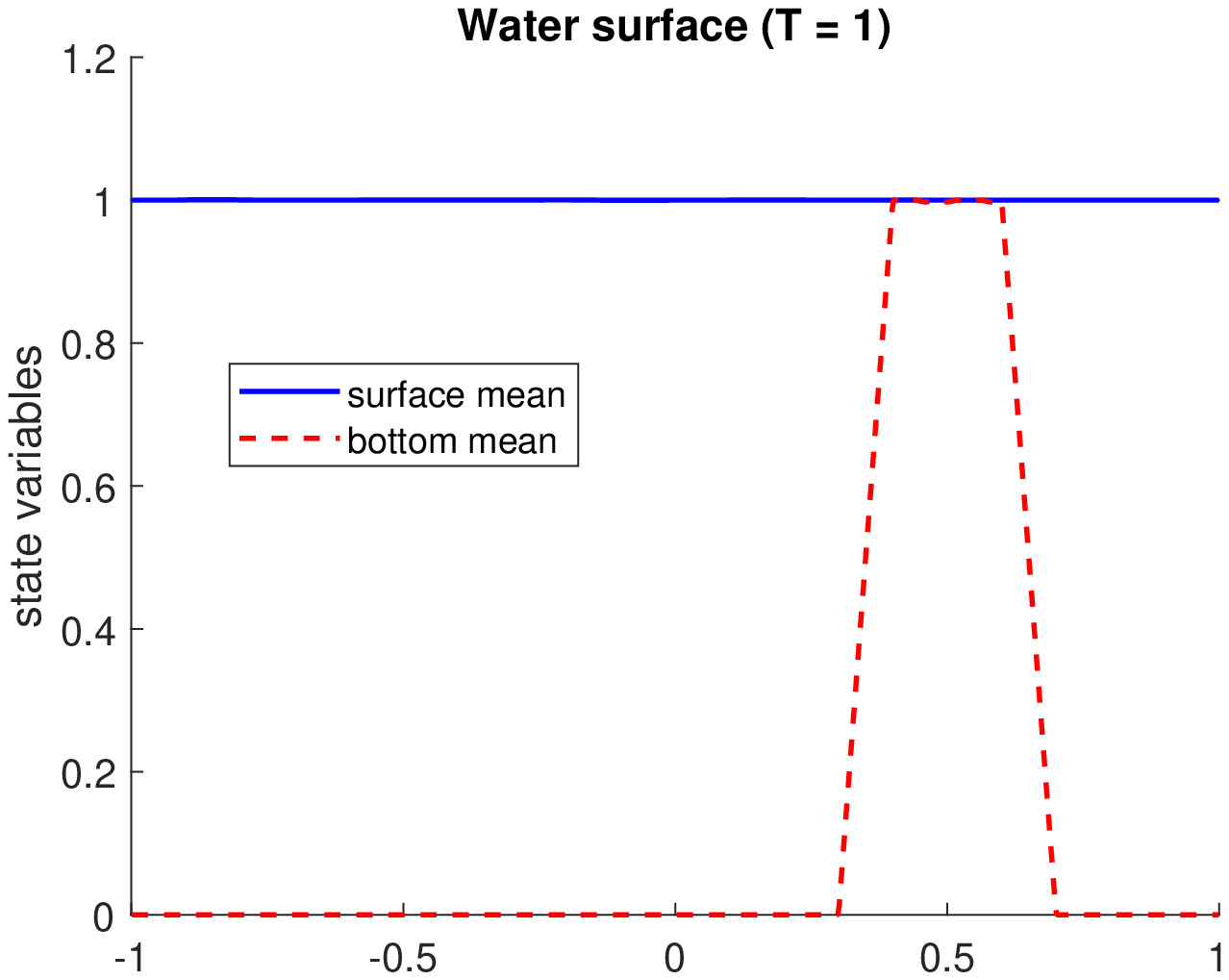}
    \hfill
    \includegraphics[width = .49\textwidth]{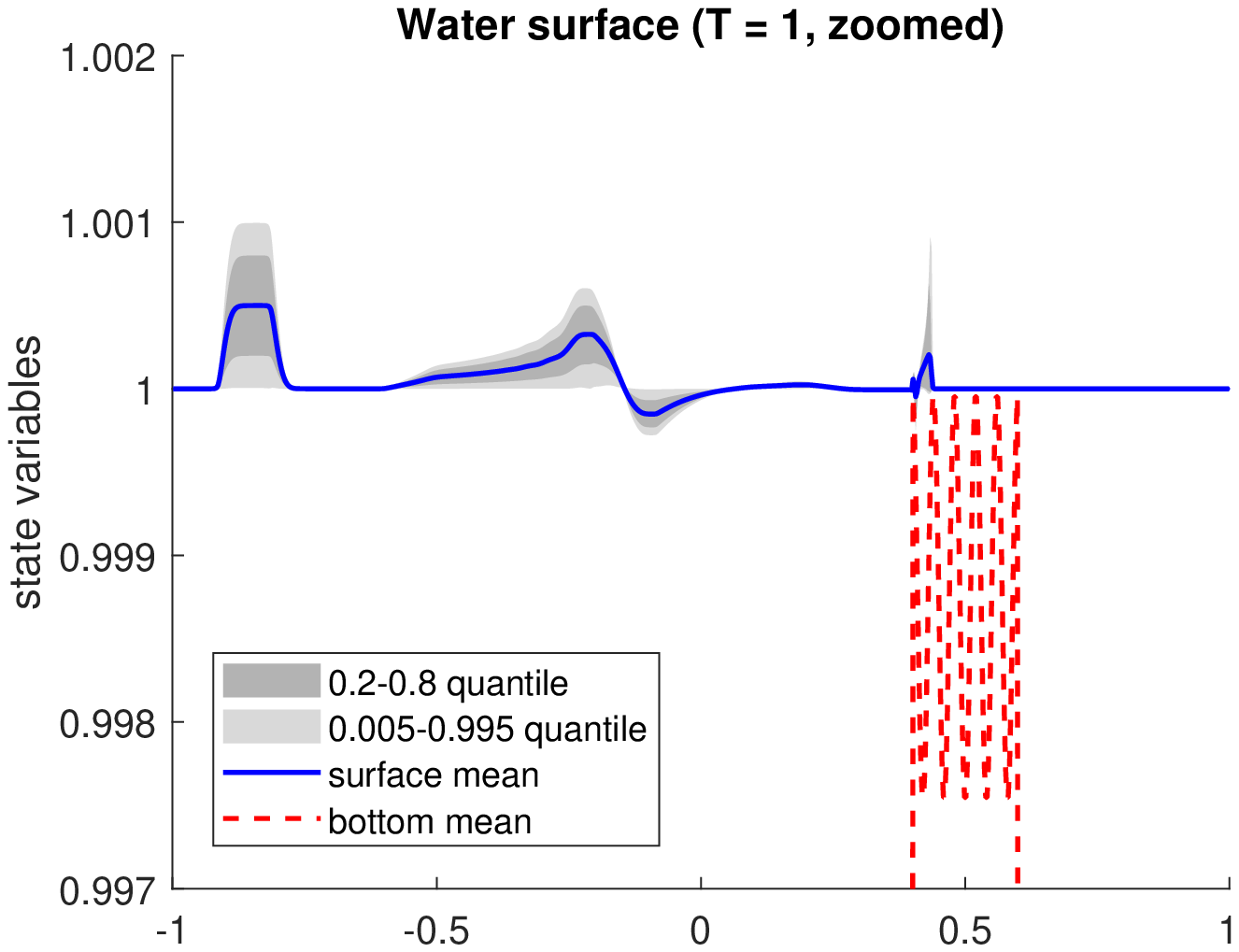}
    \hfill
    \includegraphics[width = .49\textwidth]{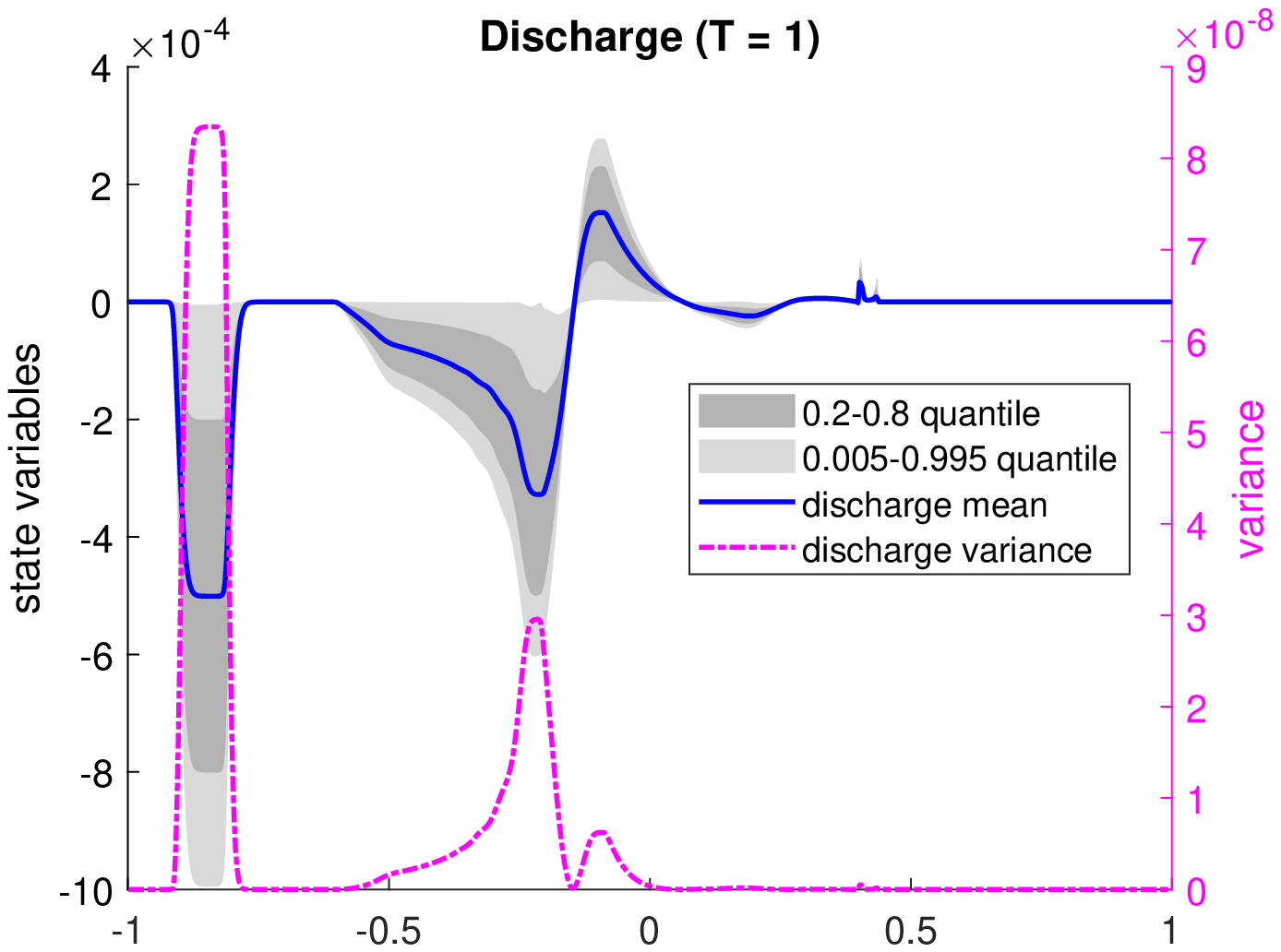}    
    \caption{Results for \cref{ssec:results-sws}: water surface (left), zoomed water surface (mid), and discharge (right) at $t=1$ for \eqref{eq:bottom 2}-\eqref{eq:IV2}, $K = 9$.}
    \label{fig:ex2}
\end{figure}

\subsection{Stochastic Discontinuous Bottom}\label{ssec:results-sdb}
For our last example, consider the shallow water system with deterministic initial conditions,
\begin{equation}\label{eq:IV3}
w(x,0,\xi) = \left\{\begin{aligned}&5.0&&x\le0.5,\\&1.6&&x>0.5,\end{aligned}\right.\qquad u(x,0,\xi) = \left\{\begin{aligned}&1.0&&x\le0.5,\\ &-2.0&&x>0.5,\end{aligned}\right.
\end{equation}
and a stochastic discontinuous bottom
\begin{equation}\label{eq:bottom 3}
B(x,\xi) = \left\{\begin{aligned}&1.5+0.1\xi&&x\le0.5,\\ &1.1+0.1\xi&&x>0.5,\end{aligned}\right.
\end{equation}
where initially we model $\xi$ as a random variable with Beta density defined by $(\alpha,\beta) = (3,1)$, which is more concentrated toward $\xi = -1$, and hence the bottom topography has higher probability of having smaller values.
At time $t = 0$, the highest possible bottom barely touches the initial water height at $x=0.5$. 
We compute the numerical solutions of $K=9$-term PCE with an $M=17$-point $\rho$-Gaussian quadrature to enforce the condition \eqref{eq:h-positivity}. We compute on a physical domain $x \in [0,1]$ with uniform cell size $\Delta x = 1/400$ up to terminal time $t=0.15$. 

In this example we observe over- and undershoots in the neighborhood of the bottom discontinuity for both the water surface $w$ and the discharge $q$ (see \cref{fig:ex3-beta24}). This phenomenon also occurs in deterministic version of \eqref{eq:IV3}-\eqref{eq:bottom 3} when numerical solutions are computed using the schemes from \cite{audusse2004fast, perthame2001kinetic}.
\begin{figure}[h]
    \centering
    \includegraphics[width = .49\textwidth]{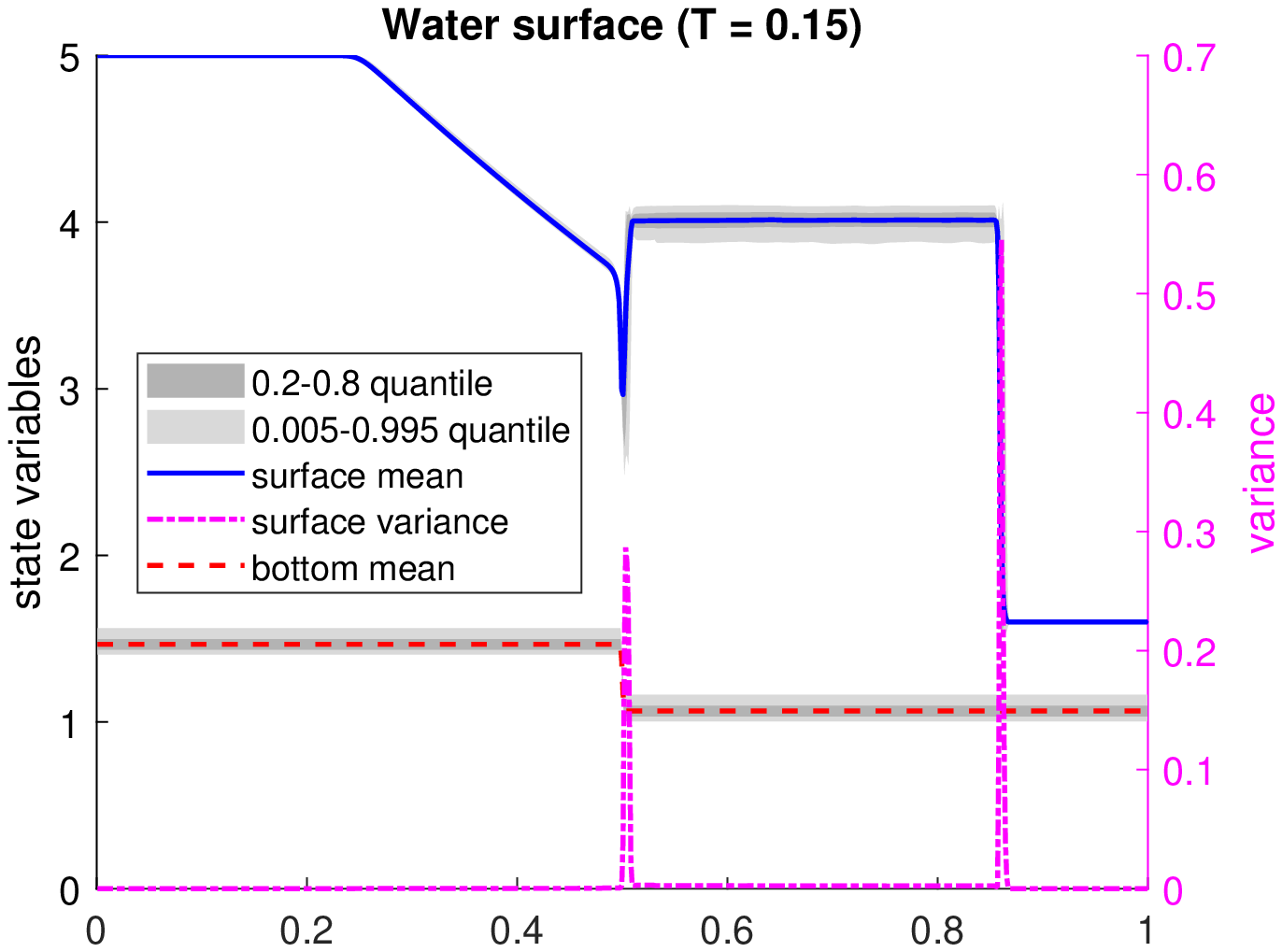}
    \includegraphics[width = .49\textwidth]{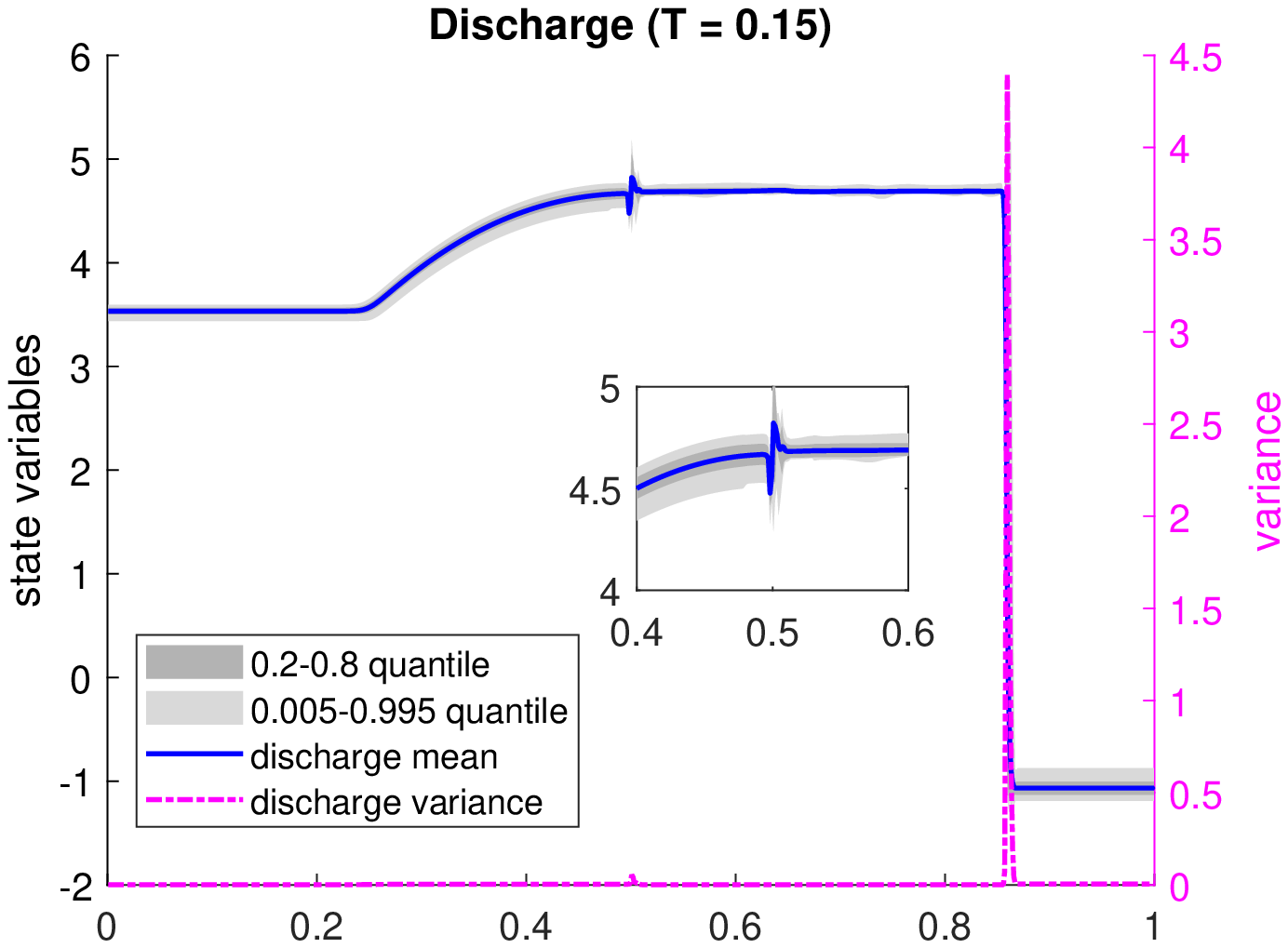}
    \caption{\cref{ssec:results-sdb} results: $K = 9$, $t = 0.15$, $(\alpha,\beta) = (3,1)$. Left figure: water surface and bottom. Right figure: discharge.}
    \label{fig:ex3-beta24}
\end{figure}
In addition we observe in this example a numerical artifact resulting from our enforcement of positivity of the water height \eqref{eq:h-positivity} at only a finite number of points: although the $99\%$ quantile region of water heights lies above 0, the $\xi$-global minimum of the water height in some cells can still be negative. Since $\mathcal{P}(\hat{h}) > 0$ only requires positivity of $h_\Lambda$ at a finite number of points, there are (low-probability) regions of the domain where the height can be negative. Note, however, that the SGSWE system is still hyperbolic and simulation can continue, despite low probability of negative water height.  

Nevertheless, the existence of negative water heights impose doubts on the applicability of the SGSWE model. Fortunately, this situation can be mitigated by increasing the number of points $M$ where positivity of $h_\Lambda$ is enforced. We observe that if the positivity of the water height is enforced at more points, the stochastic region of negative height shrinks. We demonstrate this with results in \cref{tab:table1}. In particular we observe that (a) the negative region occurs on a subinterval containing $\xi$ values greater than the maximum quadrature point, and (b) the probability of $\xi$ lying in this region is quite small.


\begin{table}
  \begin{center}
  \resizebox{0.7\textwidth}{!}{
    \renewcommand{\tabcolsep}{0.4cm}
    \renewcommand{\arraystretch}{1.3}
    {\scriptsize
      \begin{tabular}{cccc} 
      \toprule
        $M$ & $\max_m \xi_m$ & Negative Region $N_M$ & $\mathrm{Pr}[\xi \in N_M]$ \\\midrule
      $15$ & $0.934077$ & $[0.934079,1]$ & $5.75\times 10^{-6}$  \\ 
      $17$ & $0.946839$ & $[0.946899,1]$& $2.43\times 10^{-6}$\\
      $19$ & $0.956205$ & $[0.956320,1]$ & $1.12\times 10^{-6}$  \\
      $21$ & $0.963310$ & $[0.963980,1]$ & $5.18\times 10^{-7}$  \\
      \bottomrule
      \end{tabular}
    }
 }
  \end{center}
 \caption{Numerical study of $\xi$-region and associated probabilities where the water height is negative.}\label{tab:table1}
\end{table}


In a separate experiment, we also compute the numerical results when $\xi$ is modeled as random according to a $(\alpha,\beta) = (1,3)$ distribution, which is more concentrated toward $\xi = 1$.
\cref{fig:ex3-beta42} shows that at the terminal time the ``pressure'' from stochastic bottom that skews positively causes more oscillations on the water surface and the discharge compared to \cref{fig:ex3-beta24}. In this experiment, we filter both the water heights and the discharges.
\begin{figure}[h]
    \centering
    \includegraphics[width = .49\textwidth]{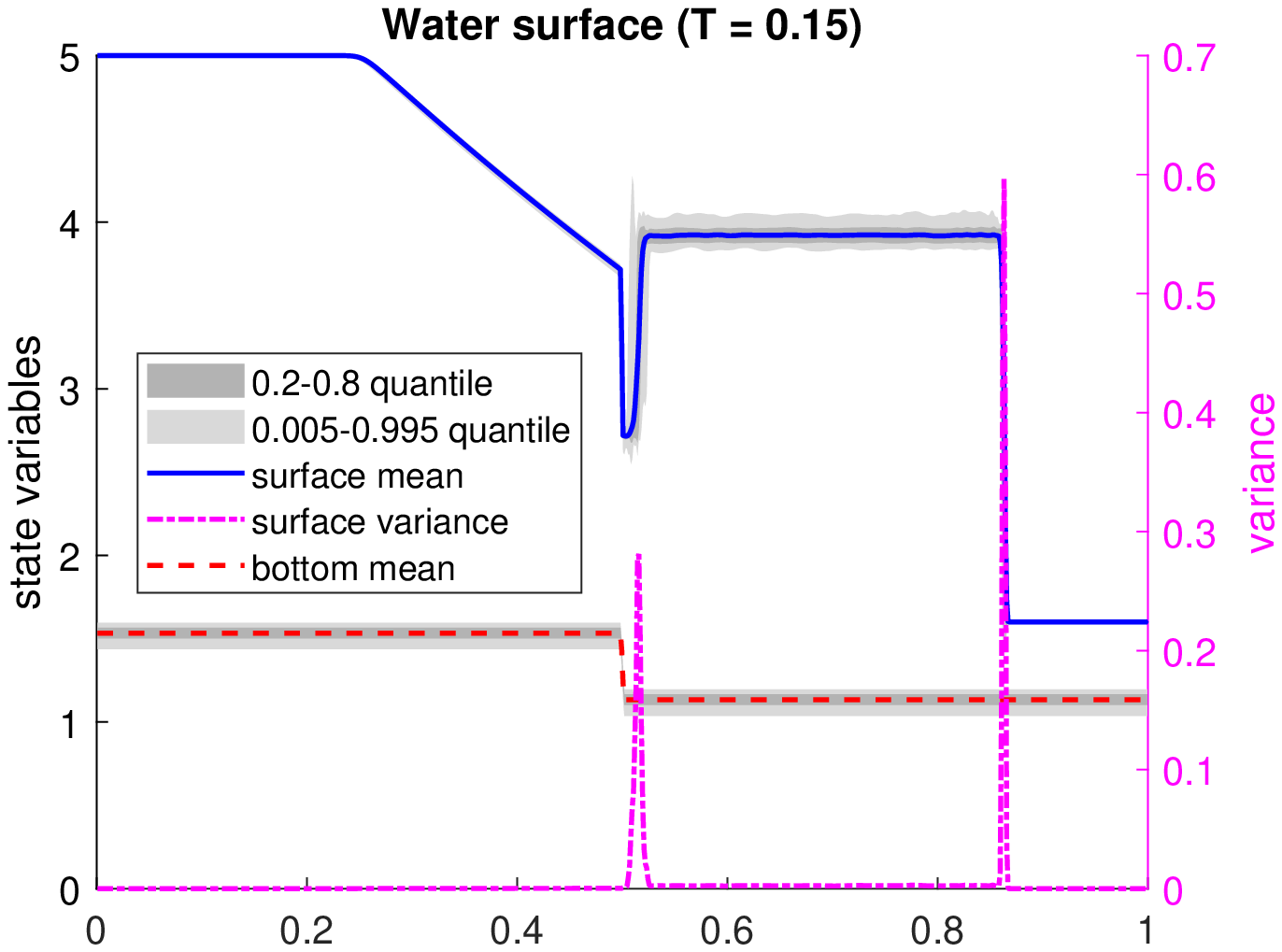}
    \includegraphics[width = .49\textwidth]{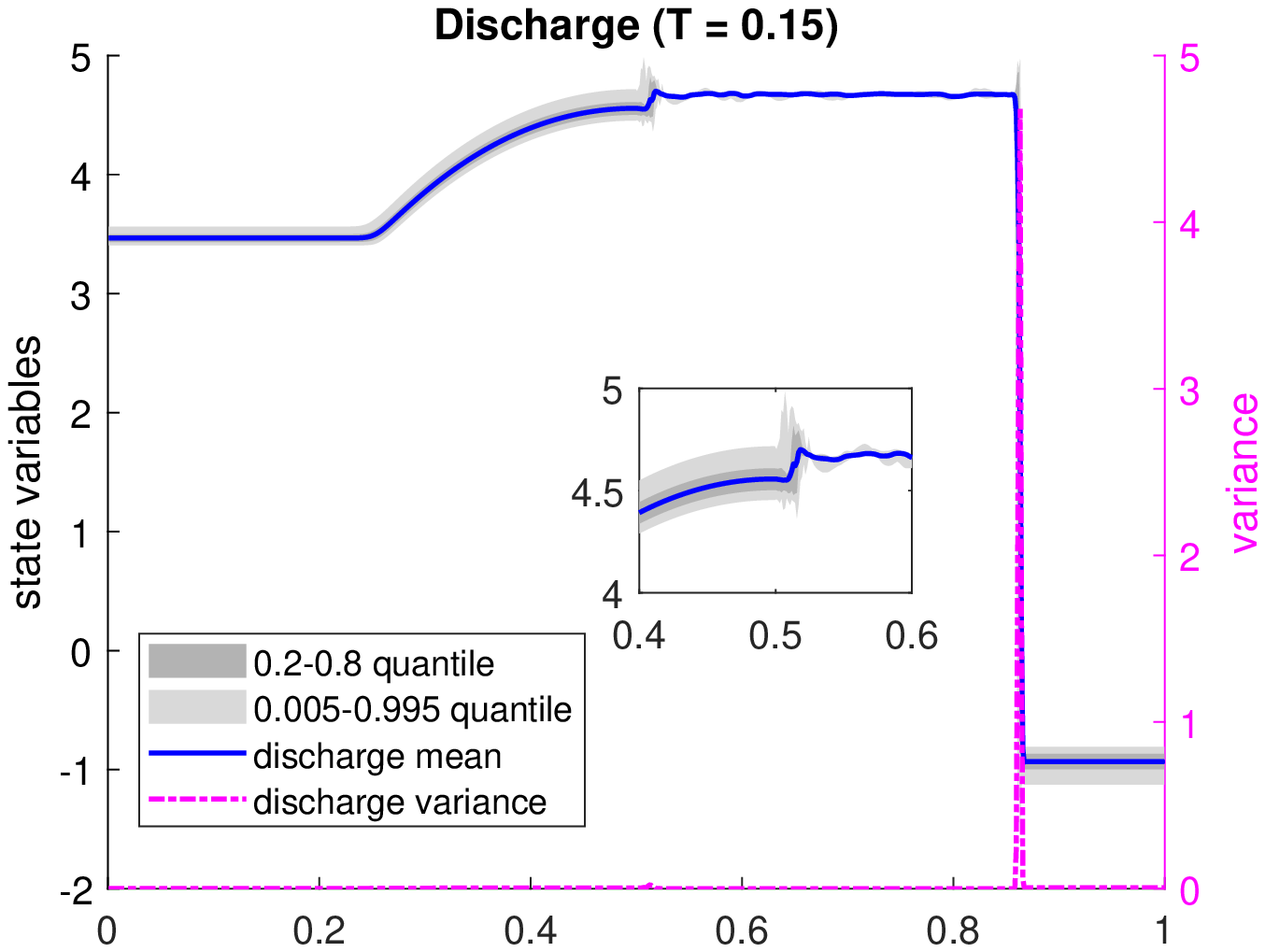}
    \caption{Numerical results with $(\alpha,\beta) = (1,3)$, $K = 9$, $t = 0.15$. Left figure: water surface and bottom. Right figure: discharge.}
    \label{fig:ex3-beta42}
\end{figure}
\appendix
\section{The Semi-Discrete Second-Order Central-Upwind \\Scheme}\label{append:cuscheme}
We briefly describe the central-upwind schemes for $1$-D balance laws. For a complete description and derivation, we refer to \cite{kurganov2001semidiscrete}. Consider the balance law,
\begin{equation}\label{eq:balancedlaw}
\mathbf{U}_t+(F(\mathbf{U}))_x = S(\mathbf{U})
\end{equation}
For a uniform mesh with cells $\mathcal{C}_i\coloneqq\left[x_{i-1/2},x_{i+1/2}\right]$ of size $|\mathcal{C}_i|\equiv \Delta x$, centered at $x_i = (x_{i-1/2}+x_{i+1/2})/2$, and assume that at certain a time level, the cell averages
\begin{equation}\label{eq:caverages}
    \overline{\mathbf{U}}^n_i \approx \frac{1}{\Delta x}\int\mathbf{U}_i(t^n)dx := \frac{1}{\Delta x}\int_{\mathcal{C}_i}\mathbf{U}(x,t^n)dx
\end{equation}
are available. The cell averages are then used to construct a non-oscillatory second-order linear piecewise reconstructions,
\begin{equation}\label{eq:reconp}
    \widetilde{\mathbf{U}}^n_i(x) = \mathbf{U}_i^n+(\mathbf{U}_x)_{i}(x-x_i),\quad x\in\mathcal{C}_i,
\end{equation}
whose slopes $(\mathbf{U}_x)_i$ are obtained by generalized minmod limiter,
\begin{equation}
    (\mathbf{U}_x)_i = \text{minmod}\left(\theta\frac{\mathbf{U}^n_{i+1}-\mathbf{U}^n_i}{\Delta x},\frac{\mathbf{U}^n_{i+1}-\mathbf{U}^n_{i-1}}{2\Delta x},\theta\frac{\mathbf{U}^n_{i}-\mathbf{U}^n_{i-1}}{\Delta x}\right),
\end{equation}
where the minmod function is defined to be 
\begin{equation*}
\text{minmod}(z_1,z_2,\cdots)\coloneqq\left\{\begin{aligned}
    &\min\{z_1,z_2,\cdots\}\quad&&\text{if } z_i>0,\forall i,\\ 
    &\max\{z_1,z_2,\cdots\}\quad&&\text{if } z_i<0,\forall i,\\ 
    &0\quad&&\text{otherwise},\end{aligned}\right.
\end{equation*}
and the parameter $\theta\in[1,2]$ controls the amount of numerical dissipation. 
The left- and right-sided reconstructions at the endpoints of $\mathcal{C}_i$ are,
\begin{equation}\label{eq:precon}
    \mathbf{U}^{+}_{i-\frac{1}{2}} = \overline{\mathbf{U}}^n_i-\frac{\Delta x}{2}(\mathbf{U}_x)_{i},\quad \mathbf{U}^{-}_{i+\frac{1}{2}} = \overline{\mathbf{U}}^n_i+\frac{\Delta x}{2}(\mathbf{U}_x)_{i}.
\end{equation}
The semidiscrete form of the central-upwind scheme is then given by,
\begin{equation}
\dfrac{d}{dt}\overline{\mathbf{U}}_i(t) = -\dfrac{\mathcal{F}_{i+\frac{1}{2}}-\mathcal{F}_{i-\frac{1}{2}}}{\Delta x}+\overline{\mathbf{S}}_i,
\end{equation}
where the numerical flux $\mathcal{F}$ and the source term $\overline{\mathbf{S}}_i$ are given in \eqref{eq:fluxcusg} and \eqref{eq:semidiscretewsg}, respectively.

\section{Proof of \cref{cor:tchakaloff}}
The Corollary is immediate from the following Lemma:
\begin{lem}\label{lemma:tch-temp}
  For some $M \leq \dim P^3_\Lambda$, there is an $M$-point positive quadrature rule that is exact on $P^3_\Lambda$. 
\end{lem}
The veracity of this lemma immediately yields $M \leq \dim P^3_\Lambda$ in \cref{cor:tchakaloff}. The second bound in that corollary results from chaining this with the dimension bound in \eqref{eq:PL}. Thus, we need only prove the above Lemma, which in turn is a simple consequence of Tchakaloff's theorem:
\begin{lem}[Tchakaloff's Theorem, \cite{bayer_proof_2006}]\label{lemma:tchakaloff}
  Let $P_{T,\ell}$ denote the space of polynomials of degree up to $\ell$ on $\R^d$:
  \begin{align*}
    P_{T,\ell} \coloneqq \mathrm{span} \left\{ \zeta^\nu \;\big|\; \sum_{J=1}^d \nu_J \leq \ell \right\}.
  \end{align*}
  Then for some $M \leq \dim P_{T,\ell}$, there exists a set of quadrature nodes $\{\zeta_m\}_{m=1}^M$ and positive weights $\{\tau_m\}_{m=1}^M$ such that 
  \begin{align*}
    \int_{\R^d} p(\zeta) \rho(\zeta) d{\zeta} &= \sum_{m=1}^M p(\zeta_m) \tau_m, & p &\in P_{T,\ell}.
  \end{align*}
\end{lem}
Now given $P^3_\Lambda$, let $\ell^\ast$ denote the maximum polynomial degree of any element in $P^3_\Lambda$:
\begin{align*}
  \ell^\ast \coloneqq \sup_{p \in P^3_\Lambda} \deg p = \max_{k=1, \ldots, K} \deg \phi_k,
\end{align*}
which is finite. Then clearly we have $P^3_\Lambda \subseteq P_{T,\ell^\ast}$.
By \cref{lemma:tchakaloff}, there is some $M^\ast \leq \dim P_{T,\ell^\ast}$ such that $\{\zeta^\ast_m\}_{m=1}^{M^\ast}$ and $\{\tau_m^\ast\}_{m=1}^{M^\ast}$ are nodes and (positive) weights, respectively, corresponding to a quadrature rule that is exact on $P_{\Lambda}$ (since it's exact on the larger set $P_{T,\ell^\ast}$). Note that if $M^\ast \leq \dim P_\Lambda^3 \eqqcolon Q$, then the result of \cref{lemma:tch-temp} is immediate, so we assume otherwise. Let $\left\{\psi_k\right\}_{k=1}^Q$ denote any basis for $P_\Lambda^3$, and define
\begin{align*}
  \mathbf{\Psi}(\zeta) \coloneqq \left[ \psi_1(\zeta), \;\; \psi_2(\zeta), \;\; \ldots \;\; \psi_Q(\zeta) \right]^T \in \R^Q.
\end{align*}
Then exactness of the quadrature rule on $P^3_\Lambda$ implies the vector-valued equality,
\begin{align*}
  \sum_{m=1}^{M^\ast} \tau^\ast_m \mathbf{\Psi}(\zeta^\ast_m) &= \mathbf{e}, & (e)_k &\coloneqq \int_{\R^d} \psi_k(\zeta) \rho(\zeta) d \zeta.
\end{align*}
I.e., $\mathbf{e} \in \R^Q$ lies in the convex hull of $\left\{ \mathbf{\Psi}(\zeta^\ast_m)\right\}_{m=1}^{M^\ast}$. By Carath\'eodory's Theorem, there must be a size-$Q$ subset of nodes $\left\{ \zeta_m\right\}_{m=1}^Q \subset \left\{ \zeta_m^\ast \right\}_{m=1}^{M^\ast}$, with positive weights $\left\{ \tau_m \right\}_{m=1}^{Q}$,
such that 
  $\sum_{m=1}^Q \tau_m \mathbf{\Psi}(\zeta_m) = \mathbf{e},$
which proves \cref{lemma:tch-temp}.

\bibliographystyle{plain}
\bibliography{bibfile}
\end{document}